\documentclass[10pt]{article}
\usepackage{amsthm,amsmath,verbatim,amssymb}

\renewcommand{\epsilon}{\varepsilon}

\newtheorem{lem}{Lemma}
\newtheorem{prop}[lem]{Proposition}
\newtheorem{thm}[lem]{Theorem}

\newtheorem{cor}[lem]{Corollary}
\newtheorem*{thm1}{Theorem A}
\newtheorem*{thm2}{Theorem B}

\theoremstyle{definition}
\newtheorem{defn}[lem]{Definition}

\theoremstyle{remark}
\newtheorem{rem}[lem]{Remark}
\newtheorem{example}[lem]{Example}

\begin{document}
\title{Filtrations and test-configurations}
\author{G\'abor Sz\'ekelyhidi \\ \\ { \it with an appendix by Sebastien Boucksom}}

\date{}

\maketitle

\begin{abstract}
	We introduce a strengthening of K-stability, based on
	filtrations of the homogeneous coordinate ring. This allows for
	considering certain limits of families of test-configurations,
	which arise naturally in several settings. We
        prove that if a manifold with no
	automorphisms admits a cscK metric, then it satisfies this
	stronger stability notion. We also discuss the relation with
	the birational transformations in the definition of
	$b$-stability. 
\end{abstract}
      
\section{Introduction}
Given a compact complex manifold $X$ with an ample line bundle $L$, 
the notion of a test-configuration is central to the definition of
K-stability, which in turn is conjecturally related to the existence of
a constant scalar curvature K\"ahler metric in the first Chern class
$c_1(L)$, by the Yau-Tian-Donaldson conjecture~\cite{Yau93,Tian97,Don02}.  
Roughly speaking, test-configurations for $(X,L)$ are 
$\mathbf{C}^*$-equivariant flat degenerations of $X$ into possibly
singular schemes. 
It was shown by Witt Nystr\"om~\cite{Ny10} that test-configurations for
$(X,L)$ give rise to 
filtrations of the homogeneous coordinate ring and in this
paper we explore
the converse direction of this. The first observation is
that every suitable filtration
gives rise to a family of test-configurations living in larger and
larger projective spaces, and that the filtration should in some sense
be thought of as the limit of this family. See Section~\ref{sec:filt}
for the detailed definitions.

It is natural to extend the
class of test-configurations to these limiting objects for several
reasons. 
For instance every convex
function on the moment polytope of a toric variety can be thought of as
a filtration, but only the rational piecewise linear convex
functions give rise
to test-configurations by Donaldson's work~\cite{Don02}. Another reason
is that Apostolov-Calderbank-Gauduchon-T\o{}nnesen-Friedman~\cite{ACGT3}
have found an example of a manifold that does not admit an extremal
metric, but does not appear to be destabilized by a test-configuration.
Rather it is destabilized by a $\mathbf{C}^*$-equivariant degeneration
which is equipped with an irrational polarization, and this 
can be thought of as a filtration. Finally in~\cite{GSz07_1} we studied
minizing sequences for the Calabi functional on a ruled surface, and
found that the limiting behavior of the metrics has an algebro-geometric
counterpart, as a sequence of test-configurations. In general there is
no limiting test-configuration, since in the sequence we need embeddings
into larger and larger projective spaces, but once again we can think of the
limit as a filtration. We will describe these examples in more detail 
in Section~\ref{sec:toric}. Note that Ross and
Witt Nystr\"om~\cite{RW11} have done related work in a more analytic
direction. Starting with a suitable filtration, they define an
``analytic test-configuration'', which is a geodesic ray in the space of
metrics in a weak sense. For more in this direction see for example
Phong-Sturm~\cite{PS06}. 

We define a notion of Futaki
invariant for filtrations, extending the usual definition. Our main result, in 
Section~\ref{sec:stoppa} is the following.
\begin{thm1}
	Suppose that $X$ admits a cscK metric in
	$c_1(L)$, and the automorphism group of $(X,L)$ is finite. Then
	if $\chi$ is a filtration for $(X,L)$ such that
	$\Vert\chi\Vert_2 > 0$, then the Futaki invariant of $\chi$
	satisfies $\mathrm{Fut}(\chi) > 0$. 
\end{thm1}
\noindent Here $\Vert\chi\Vert_2$ is a norm of the filtration, and the
filtrations with zero norm play the role of the trivial
test-configuration.
This result is a strengthening
of Stoppa's result~\cite{Sto08}, whose conclusion under the same
assumptions is that 
$(X,L)$ is K-stable, 
since it implies that
the Futaki invariant has to be bounded away from zero uniformly along
certain families of test-configurations. In addition, similarly to
Stoppa's argument, we use the existence result for cscK metrics on blowups due
to Arezzo-Pacard~\cite{AP06}, and the asymptotic Chow stability of
cscK manifolds with no discrete automorphism group due to
Donaldson~\cite{Don01}. 

A key new ingredient in the proof is the Okounkov
body~\cite{Ok96}, and the concave (in our case convex) transform of a
filtration introduced by Boucksom-Chen~\cite{BC09}, which was also used
in the context of test-configurations 
by Witt Nystr\"om~\cite{Ny10}. We review these
constructions in Section~\ref{sec:Okounkov}. 

In addition, the proof relies on the following result, which was stated as a
conjecture in an earlier version of this paper. The result is due to
S. Boucksom, and the proof is presented in the appendix as Theorem~\ref{thm:sz}. 
\begin{thm2}
	Suppose that $S\subset\bigoplus_{k\geqslant 0} H^0(X,L^k)$ 
	is a graded subalgebra which contains an ample series (see
	Definition~\ref{defn:ample}). In addition suppose that 
	\[ \lim_{k\to\infty} k^{-n}\dim S_k < \lim_{k\to\infty}
	k^{-n}\dim H^0(X,L^k),\]
	where $n$ is the dimension of $X$. Then there is a point $p\in
	X$ and a number $\epsilon > 0$, such that
	\[ S_k \subset H^0(X, L^k\otimes I_p^{\lceil k\epsilon
	\rceil}),\]
	for all $k$, where $I_p$ is the ideal sheaf of the point $p$. 
\end{thm2}

In~\cite{Don10} Donaldson introduced a new notion of stability, called
$b$-stability, which is a similar strengthening of $K$-stability, but it
allows for more general families of test-configurations (and
even more general degenerations) than what we are able to encode using
filtrations so far. In
Section \ref{sec:bstability} we make some basic
observations about the relation with filtrations. In particular we will
show that 
Proposition~\ref{prop:bstab}, which is a variant Theorem A
above,  gives a strengthening  of the main theorem
in~\cite{Don11}. 

\subsection*{Acknowledgements}
I would like to thank Jeff Diller, Simon Donaldson, Sonja Mapes and Jacopo Stoppa for
useful conversations. I am
also grateful for Sebastien Boucksom providing the proof of
Theorem B as an appendix to this paper. 
This work was
partially supported by NSF grant DMS-0904223.

\section{Test-configurations, the Futaki invariant and the Chow weight}
We briefly recall the notion of test-configuration and their Futaki
invariants from
Donaldson~\cite{Don02}. Given a polarized variety $(X,L)$, a
test-configuration 
for $(X,L)$ is a flat, polarized, $\mathbf{C}^*$-equivariant
family $(\mathcal{X},\mathcal{L})\to\mathbf{C}$, where the generic fiber
is isomorphic to $(X,L^r)$ for some $r > 0$. The number $r$ is called
the exponent of the test-configuration. The Futaki invariant and
the Chow weight are both computed in terms of the induced
$\mathbf{C}^*$-action on the central fiber $(X_0, L_0)$. Namely let us
write $d_{rk}$ for the dimesion of, and $w_{rk}$ for the total weight of the
action on $H^0_{X_0}(L_0^k)$. For large $k$ we have expansions
\begin{equation}\label{eq:expand} \begin{aligned}
	d_{rk} = a_0(rk)^n + a_1(rk)^{n-1} + \ldots \\
	w_{rk} = b_0(rk)^{n+1} + b_1(rk)^n + \ldots,
\end{aligned}
\end{equation}
where $n$ is the dimesion of $X$. We write the expansions in terms of
$rk$ instead of $k$, because we think of the numbers $d_{rk}$ and
$w_{rk}$ as being related to the
line bundles $L^{rk}$ on $X$. For instance this way the number $a_0$ is
the volume of $(X,L)$, and does not depend on the exponent $r$ of the
test-configuration. The Futaki invariant of the family is
defined to be
\[ \mathrm{Fut}(\mathcal{X},\mathcal{L}) = \frac{a_1b_0 - a_0b_1}{a_0^2}.\]
Note that the Futaki invariant remains unchanged if we replace the line
bundle $\mathcal{L}$ on $\mathcal{X}$ by a power. 
The Chow weight of the family is 
\begin{equation}\label{eq:Chow}
	\mathrm{Chow}_r(\mathcal{X},\mathcal{L}) = \frac{rb_0}{a_0} -
\frac{w_r}{d_r}.
\end{equation}
In the notation for the Chow weight, the subscript $r$ means that the
test-configuration has exponent $r$. We emphasize this, 
since unlike for the Futaki invariant, it makes a
difference if we replace $\mathcal{L}$ by a power, and later on we will
not have the line bundle explicit in the notation. In fact we have
\[ \mathrm{Chow}_{rk}(\mathcal{X},\mathcal{L}^k) = \frac{krb_0}{a_0} -
\frac{w_{kr}}{d_{kr}},\]
from which it is easy to check that
\begin{equation}\label{eq:chowinf}
	\mathrm{Fut}(\mathcal{X}) = \lim_{k\to\infty}
	\mathrm{Chow}_{rk}(\mathcal{X},\mathcal{L}^k).
\end{equation}
For the record we state the following definitions (see for example
Ross-Thomas~\cite{RT04}). 
\begin{defn}
	The polarized manifold $(X,L)$ 
	is $K$-stable, if the Futaki invariant is
	positive for every test-configuration, for which the central
	fiber is not isomorphic to $X$. 

	The polarized manifold $(X,L)$ is asymptotically
	Chow stable, if there is some $k_0$, such that
	the Chow weight is positive for all
	test-configurations with exponent greater than
	$k_0$, and whose central fiber is not isomorphic
	to $X$. 
\end{defn}

We will need to define a norm for test-configurations. There are various
options for this, analogous to various $L^p$ norms for functions. Given
a test-configuration as above, write $A_{rk}$ for the generator of the
$\mathbf{C}^*$-action on $H^0_{X_0}(L_0^k)$. So
$\mathrm{Tr}(A_{rk})=w_{rk}$ in our notation above. We then have an
expansion
\begin{equation}\label{eq:TrA2}
	\mathrm{Tr}(A_{rk}^2) = c_0(rk)^{n+2} + \ldots
\end{equation}
for large $k$, and we define the norm $\Vert\mathcal{X}\Vert_2$ 
of the test-configuration by
\begin{equation}\label{eq:normchi}
	\Vert\mathcal{X}\Vert_2^2 = c_0 - \frac{b_0^2}{a_0}.
\end{equation}
This is analogous to the $L^2$-norm of functions, normalized to be zero
on constants. Note that the norm is
unchanged if we replace $\mathcal{L}$ by a power. 

In what follows, it will be natural to think of test-configurations
slightly differently. Recall that all test-configurations of exponent $r$
for $(X,L)$ can be obtained by embedding $X\hookrightarrow
\mathbf{P}(V^*)$ for $V=H^0(X,L^r)$, and then choosing a
$\mathbf{C}^*$-action on $V^*$. The test-configuration is then obtained
by taking the $\mathbf{C}^*$-orbit of $X$, and completing this family
across the origin with the flat limit. 
Let us assume that the weights of the dual action on $V$ are
all positive (we can modify the original $\mathbf{C}^*$-action by
another action with constant weights, without changing any of the
invariants of the
test-configuration). The weight decomposition under this
$\mathbf{C}^*$-action gives rise to a flag
\begin{equation}\label{eq:flag}
	\{0\} = V_0 \subset V_1 \subset\ldots\subset V_k = V, 
\end{equation}
where $V_i$ is spanned by the eigenvectors with weight at most $i$. The
point we want to make is that the test-configuration is determined by
this flag. This can be seen as follows. Suppose that
$\lambda_1,\lambda_2 : \mathbf{C}^*\to GL(V)$ are two one-parameter
subgroups, with the same flag \eqref{eq:flag}. Let $v\in V$ be such
that $\lambda_1(t)\cdot v = t^iv$ for all $t$, and let
$v=w_1+\ldots+w_i$ be the weight decomposition of $v$ with respect to
$\lambda_2$. Note that only weights up to $i$ occur in this
decomposition since $\lambda_2$ has the same flag as $\lambda_1$. It
follows that
\[ \lambda_2(t)^{-1}\lambda_1(t)\cdot v =
t^i(t^{-1}w_1+\ldots+t^{-i}w_i), \]
and so
\[ \lim_{t\to 0} \lambda_2(t)^{-1}\lambda_1(t)\cdot v = w_i. \]
Applying this to each weight vector for $\lambda_1$, we see
that $M(t) = \lambda_2(t)^{-1}\lambda_1(t)$ extends to a map
$M : \mathbf{C}\to GL(V)$ (the fact that $M(0)$ is invertible follows by
interchanging $\lambda_1,\lambda_2$ in the above argument). It then
follows that the families in $\mathbf{P}(V^*)$ defined by the orbits of
$X$ under the dual actions of $\lambda_1$ and $\lambda_2$ are equivalent.
Because of this, we will often speak of the test-configuration induced
by a flag in $H^0(X,L^r)$, and also we will make use of the matrices
$A_k$ as above, as if we have already picked a $\mathbf{C}^*$-action
giving rise to the flag. The point of view of flags is useful more
generally  in GIT, see for example
Section 2.2 in Mumford-Fogarty-Kirwan~\cite{MFK94}. 

\section{Filtrations}\label{sec:filt}
Let $(X,L)$ be a polarized manifold. Let us write $R_k = H^0(X,L^k)$,
and
\[ R = \bigoplus_{k\geqslant 0} R_k = \bigoplus_{k\geqslant 0} H^0(X,L^k)\]
for the homogenenous coordinate ring of $(X,L)$. We will assume
throughout the paper that
$R_1$ generates $R$. 
\begin{defn}\label{defn:filtration}
	A \emph{filtration} of $R$ is a chain of finite dimensional subspaces
	\[ \mathbf{C} = F_0R \subset F_1R \subset F_2R \subset \ldots
	\subset R,\]
	such that the following conditions hold:
	\begin{enumerate}
		\item The filtration is multiplicative, i.e. $(F_iR)(
			F_jR) \subset F_{i+j}R$ for all $i,j\geqslant 0$,
		\item The filtration is compatible with the grading $R_k$ of
			$R$, i.e. if $f\in F_iR$ for some $i\geqslant 0$
			then each homogeneous
			piece of $f$ is in $F_iR$,
		\item We have
			\[ \bigcup_{i\geqslant 0} F_iR = R.\]
	\end{enumerate}
\end{defn}

This notion of filtration is more or less equivalent to the one used in
Witt Nystr\"om~\cite{Ny10}. The main difference is that our indices are
the negative of his, and in addition our filtration is ``scaled'' so
that each nontrivial piece has positive index. In analogy to~\cite{Ny10}
we could allow more general filtrations, where $F_iR$ can be non-empty
for negative $i$ as well, assuming a boundedness condition. 
Namely we assume that for some constant $C$, the
filtration $F_iR_k$ on the degree $k$ piece of $R$ satisfies
$F_{-Ck}R_k=\{0\}$. In this case we could define a new filtration by
letting $F'_iR_k = F_{i-Ck}R_k\oplus\mathbf{C}$ for all $i\geqslant 0$,
and it would satisfy our conditions. In addition in \cite{Ny10} the
filtered pieces are indexed by real numbers, while ours are integers,
but this is also not a significant restriction.  

Given a filtration $\chi$ of $R$, the Rees algebra of
$\chi$ is defined by 
\[ \mathrm{Rees}(\chi) = \bigoplus_{i\geqslant 0} (F_iR)t^i \subset
R[t].\]
This is a flat $\mathbf{C}[t]$-subalgebra of $R[t]$, since it is a
torsion-free $\mathbf{C}[t]$-module (see Corollary 6.3 in
Eisenbud~\cite{Eis95}). In addition
the associated graded algebra of $\chi$ is
\[ \mathrm{gr}(\chi) = \bigoplus_{i\geqslant 0}
(F_iR)/(F_{i-1}R),\]
where $F_{-1}R = \{0\}$. Note that both of these algebras have
two gradings. One grading comes from the grading of $R$, while
another, denoted by $i$ here, comes from the filtration. 
The fiber of the Rees algebra of $\chi$ at non-zero $t$ is
isomorphic to $R$,
while the fiber at $t=0$ is isomorphic to $\mathrm{gr}(\chi)$. 

\subsection{Finitely generated filtrations}\label{sec:fgfilt}
Let us call a filtration
finitely generated, if its Rees algebra is finitely generated. 
In this case the filtration gives rise to a
test-configuration for $(X,L)$, whose total space is
$\mathrm{Proj}_{\mathbf{C}[t]}\mathrm{Rees}(\chi)$,
where the grading in the $\mathrm{Proj}$ construction is the grading
coming from $R$ (which is supressed in the notation).  The central fiber
of the test-configuration is
$\mathrm{Proj}_\mathbf{C}(\mathrm{gr}(\chi))$, where again we are using
the grading induced by the grading of $R$. The grading given by the
filtration is the one which induces a $\mathbf{C}^*$-action
on the family as well as on its central fiber. In order for the action to
be compatible with multiplication on $\mathbf{C}$, the function $t$
must have weight $-1$. This implies that in terms of sections on the
central fiber, the sections in $(F_iR)/(F_{i-1}R)$ have weight $-i$.
It is these weights that are used in the calculation of the Futaki
invariant.    

Finitely generated filtrations therefore give rise to test-configurations. 
Conversely, Witt Nystr\"om~\cite{Ny10} showed that 
every test-configuration gives rise to a finitely generated
filtration of $R$. Let us
recall  the
construction briefly. We are thinking of a test-configuration as a
$\mathbf{C}^*$-equivariant flat family
$\pi:(\mathcal{X},\mathcal{L})\to\mathbf{C}$, such that the generic
fiber is isomorphic to $(X,L^r)$ for some power $r > 0$. If $s\in
R_{r}$, then we can think of $s$ as a section of
$\mathcal{L}$ over the fiber $\pi^{-1}(1)$. Using the
$\mathbf{C}^*$-action we can extend $s$ to a meromorphic section
$\overline{s}$ of $\mathcal{L}$ over the whole of $\mathcal{X}$. We
then define
\begin{equation}\label{eq:filt1}
	F_iR_{r} = \{ s\in R_{r}\,:\, t^i\overline{s}\text{ is holomorphic
on }\mathcal{X}\}.
\end{equation}
Note that Witt Nystr\"om uses $t^{-i}\overline{s}$ instead of
$t^i\overline{s}$, so his filtration is the opposite of ours. This
filtration may not satisfy that $F_0R_r$ is empty (which we require of
our filtrations), but this can easily be achieved by first modifying the
$\mathbf{C}^*$-action on $\mathcal{L}$ by an action with constant
weights. 
We can then extend this filtration of $R_r$ to a filtration
of $R$ as follows. Let $N$ be such that $F_NR_r=R_r$. Then let
$\mathcal{R}\subset R[t]$ be the $\mathbf{C}[t]$-subalgebra generated by
\begin{equation}\label{eq:filt3}
	R_1t^N\oplus\left( \bigoplus_{i=1}^N (F_i R_r)t^i\right).
\end{equation}
We can then define a filtration 
\begin{equation}\label{eq:filt2}
	F_iR = \{ s\in R\, :\, t^is\in \mathcal{R}\}. 
\end{equation}
The point of adding in the generators $R_1t^N$ is to ensure that for
every $s\in R$ there is some $i$ such that $s\in F_iR$, i.e. that
Condition (3) in Definition~\ref{defn:filtration} holds. At the same
time because of the choice of $N$, the induced filtration on $R_{kd}$
for any $k > 0$ conincides with that obtained by the construction in
Equation \eqref{eq:filt1} applied to sections of $\mathcal{L}^k$. It
follows from this that $\mathrm{Proj}_{\mathbf{C}[t]}\mathcal{R}$ is
isomorphic to the test-configuration $\mathcal{X}$ that we started with.

\subsection{General filtrations}\label{sec:genfilt}
The main point of considering
filtrations instead of test-configurations is that filtrations are
more general, since they are not all finitely generated. 
At the same time any 
filtration can be approximated by finitely generated filtrations in
the following sense. Suppose that $\mathcal{R}$ is the Rees algebra
corresponding to a filtration $\chi$, and in addition let
$\mathcal{R}_i$ be a sequence of finitely generated
$\mathbf{C}[t]$-subalgebras of $\mathcal{R}$, such that
\[ \mathcal{R}_1\subset\mathcal{R}_2\subset \ldots \subset\mathcal{R},\]
and $\bigcup_{i > 0}\mathcal{R}_i = \mathcal{R}$. Then using the
construction in Equation~\eqref{eq:filt2} we obtain a family of induced
filtrations $\chi_i$, and we think of $\chi$ as the limit of the
sequence $\chi_i$. 

Given a filtration $\chi$ it will be convenient to choose one specific
approximating sequence $\chi^{(k)}$. Namely for each $k$ we let
$\chi^{(k)}$ be
the finitely generated filtration induced by the filtration on $R_k$
given by $\chi$, exactly as above, in Equations \eqref{eq:filt3} and
\eqref{eq:filt2}. Equivalently, we can think of $\chi^{(k)}$ as the
test-configuration of exponent $k$, corresponding to the filtration on
$R_k$ as we described at the end of the last section.

We will use the following comparison between $\chi^{(k)}$ and $\chi$
many times. For any $l$, let us write $F_i'R_{kl}$ and $F_iR_{kl}$ for
the filtrations on $R_{kl}$ given by $\chi^{(k)}$ and $\chi$
respectively. Then by construction $F_i'R_k = F_iR_k$ for all $i$, and
$F_i'R_{kl}\subset F_iR_{kl}$ for $l > 1$. Indeed, once we fix the
filtration $\chi^{(k)}$ on $R_k$, then for all $l > 1$ and $i$,
the space $F_i'R_{kl}$ is the smallest possible subspace of $R_{kl}$,
which is compatible with the multiplicative property of $\chi^{(k)}$.

\begin{defn}\label{defn:filtFut}
	Given a filtration $\chi$, we define the Futaki invariant, and
	$k^\text{th}$ Chow weight of
	$\chi$ to be
	\[ \begin{aligned}
		\mathrm{Fut}(\chi) &= \liminf_{k\to\infty}
		\mathrm{Fut}(\chi^{(k)},\mathcal{L}) \\
		\mathrm{Chow}_k(\chi) &=
		\mathrm{Chow}_k(\chi^{(k)},\mathcal{L}),
	\end{aligned}\]
	where $(\chi^{(k)},\mathcal{L})$ 
	is the test-configuration of exponent $k$
	defined by the filtration
	on $R_k$ induced by $\chi$. We also define a norm of the
	filtration by
	\[ \Vert\chi\Vert_2 = \liminf_{k\to\infty} \Vert\chi^{(k)}\Vert_2.\]
	We will see in Lemma~\ref{lem:norm}
	that the $\liminf$ in the definition of the norm is actually a limit. 
\end{defn}

There are other possible numerical invariants of a filtration, related
to the Futaki invariant. For instance in Donaldson's work~\cite{Don11}
the relevant quantity is the asymptotic Chow weight of a filtration,
which is $\liminf_{k\to\infty} \mathrm{Chow}_k(\chi)$. We will explain
this in Section~\ref{sec:b-stab}. Note that if the filtration is
finitely generated, then the asymptotic Chow weight is equal to the
Futaki invariant, because of Equation \eqref{eq:chowinf}.

\begin{example} \label{ex:filt1}
	For filtrations, the role of trivial test-configurations is
	played by filtrations with zero norm. 
	This includes filtrations
which are limits of non-trivial test-configurations. 
For example on $\mathbf{P}^1$, we can
define the filtration (where $R_k = H^0(\mathcal{O}(k))$)
\[ F_iR_k = \{\text{all sections vanishing at }(0:1)\},\]
for $0 < i < k$, and 
\[ F_iR_k = R_k,\]
for $i \geqslant k$. It is not hard to check that the norm of this
filtration is 0. The corresponding sequence of test-configurations is
simply deformation to the normal cone of the point $(0:1)$, with smaller and
smaller parameters as $k\to \infty$ (see Ross-Thomas~\cite{RT06}). 
While none of these test-configurations is trivial, it
is reasonable that their limit should be thought of as being trivial,
and in particular the Futaki invariant of this filtration is zero.  
\end{example}

\begin{example}\label{ex:trivial}
On the other hand there are also non-trivial test-configurations which
have zero norm. For example the test-configuration for $\mathbf{P}^1$,
whose central fiber is a double line (i.e. the family of conics
$z^2-txy=0$ as $t\to 0$) has zero norm, even
though it has non-zero Futaki invariant. Note that after 
taking the normalization of the total space, 
the test-configuration becomes a product configuration.\end{example}

We say that a filtration $\chi$ is \emph{destabilizing}, if
$\Vert\chi\Vert_2 > 0$, and $\mathrm{Fut}(\chi)\leqslant 0$. We
expect that if $X$ admits a cscK metric in the class
$c_1(L)$ and has no holomorphic vector fields, then no destabilizing
filtration exists. This is a slightly stronger statement than saying
that $(X,L)$ is K-stable, since certain limiting objects are also
required to have positive Futaki invariant. On the other hand the
condition $\Vert\chi\Vert_2 > 0$ does exclude some non-trivial
test-configurations which are considered in K-stability, 
like the one in Example~\ref{ex:trivial}. At the same time 
it was pointed out by
Li-Xu~\cite{LX11} that even in the definition of K-stability one should
not consider test-configurations such as these by restricting attention to
test-configurations with normal total space. The reason is that there
are always certain non-normal test-configurations, which are
non-trivial, but have zero Futaki invariant. We therefore believe that
the condition $\Vert\chi\Vert_2 > 0$ is very natural even for
test-configurations.

\section{Examples}\label{sec:toric}
For toric varieties Donaldson~\cite{Don02} showed that any
rational piecewise linear convex function on the moment polytope
gives rise to a test-configuration of the variety. We will show that 
at the same time any positive 
convex function on the polytope gives rise
to a filtration of the homogeneous coordinate ring. Since adding a
constant to a rational piecewise linear convex function only changes the
test-configuration by an action on the line bundle with constant
weights, it is not restrictive to only consider positive functions. 

Suppose that $f:\Delta \to
\mathbf{R}$ is a positive convex function, 
where $\Delta$ is the moment polytope
corresponding to the polarized toric variety $(X,L)$. For us $\Delta$ is
closed, so $f$ is automatically bounded, although in Donaldson's
work~\cite{Don02} some unbounded convex functions also play a role. At
the same time we can allow functions which are not continuous at the
boundary of $\Delta$. 
A basis of
sections of $H^0(X,L^k)$ can be identified with the rational
lattice points in $\Delta\cap \frac{1}{k}\mathbf{Z}^n$. If 
\[ \alpha\in \Delta\cap \frac{1}{k}\mathbf{Z}^n,\]
write $s_\alpha$ for the corresponding section of $L^k$. Now on
$R_k=H^0(X,L^k)$ define the filtration as follows:
\begin{equation}\label{FiRk}
	F_i R_k = \mathrm{span}\left\{ s_\alpha\,:\, kf(\alpha)
\leqslant i\right\}.
\end{equation}
The convexity of $f$ ensures that the filtration of the graded ring of
$(X,L)$ defined in this way will satisfy the multiplicative property.
The other two conditions in Definition~\ref{defn:filtration} also
follow easily. 

We can also see what the sequence of test-configurations are, which
approximate the filtration defined by $f$. 
Let $f_k:\Delta\to\mathbf{R}$ be the largest
convex function which on the points
$\alpha\in\Delta\cap\frac{1}{k}\mathbf{Z}^n$ is defined by
\[ f_k(\alpha) = \frac{1}{k}\lceil kf(\alpha)\rceil.\] 
Then the filtration defined on $R_k$ by \eqref{FiRk} 
using the function $f$ is the same as that obtained by the same formula,
but using the function $f_k$. So the test-configuration obtained 
from the filtration
on the piece $R_k$ can be seen as the toric test-configuration defined
by the function $f_k$, which is a rational piecewise-linear
approximation to the function $f$. As for the Futaki invariants,
Donaldson showed that the test-configuration corresponding to $f_k$ has
Futaki invariant up to a constant factor given by
\[ \mathrm{Fut}(f_k) = \int_{\partial\Delta} f_k\,d\sigma - a\int_\Delta
f_k\,d\mu,\]
where $d\sigma$ is a certain measure on the boundary, and $a$ is a
normalizing constant ($a=a_1/a_0$ in the notation of
Equation~\eqref{eq:expand}). Since $f_k$ is a decreasing sequence of
functions converging to $f$ pointwise, we have
\[ \lim_{k\to\infty} \mathrm{Fut}(f_k) = \int_{\partial\Delta}
f\,d\sigma - a\int_\Delta f\,d\mu.\]
In~\cite{Don02} this functional plays an important role even when
defined on convex
functions which are not piecewise linear. It is therefore useful that
it can still be interpreted algebro-geometrically, as the Futaki
invariant of a non-finitely generated filtration. 

Another instance where more general convex functions appear is in the
study of optimal test-configurations for toric varieties~\cite{GSz07_2}.
Note that the optimal destabizing convex functions constructed in that
paper are not known to be bounded, so the filtration given by
Equation~\ref{FiRk} might not satisfy Condition (3) in
Definition~\ref{defn:filtration}. We hope 
that with more work one can show that the optimal
destabilizing convex functions are actually bounded, but in any case
this filtration should 
be thought of as being 
analogous to the Harder-Narasimhan filtration of an unstable vector
bundle. It is tempting to speculate that in general, on any
unstable manifold 
$(X,L)$ one can define such an optimal destabilizing filtration.

This picture can be extended to bundles of toric varieties, in
particular to ruled surfaces, following~\cite{GSzThesis}. 
In this way, the ``optimal destabilizing
test-configurations'' that we found in~\cite{GSz07_1} can also be seen
as filtrations. In addition
Apostolov-Calderbank-Gauduchon-T\o{}nnesen-Friedman~\cite{ACGT3} found
an example of a $\mathbf{P}^1$-bundle over a 3-fold
that does not admit an extremal metric, but
appears to be only destabilized by a non-algebraic degeneration (it has
not been shown that there are no destabilizing test-configurations).
This also fits into the above picture applied to toric bundles, and thus
can also be thought of as a filtration.

\section{The Okounkov body}\label{sec:Okounkov}
The Okounkov body~\cite{Ok96} is a convenient way to package some 
information about the graded ring $R$ and its filtrations, as shown by
Boucksom-Chen~\cite{BC09}, and Witt Nystr\"om~\cite{Ny10}. In this section we
briefly recall the main points of this, but see~\cite{BC09} and also
Lazarsfeld-Musta\c{t}\v{a}~\cite{LM09} for more details.  

First we recall the construction of the Okounkov body.
Choose a point $p\in X$ and a set of local holomorphic coordinates
$z_1,\ldots,z_n$ centered at $p$. Let $s\in H^0(X,L)$ be a section which
does not vanish at $p$. Then every section $f\in H^0(X,L^k)$ can be
written near $p$ as
\begin{equation}\label{eq:section}
	f = s^k\cdot(\text{power series in }z_1,\ldots,z_n).
\end{equation}
We use the graded lexicographic order on monomials. This means that
monomials with larger total degree are larger, and monomials with the
same degree are ordered using the lexicographic order. Writing
$R=\bigoplus H^0(X,L^k)$, we
can define a map
\[ \nu : R\mapsto \mathbf{Z}^n,\]
such that $\nu(f)$ is equal to the exponent of the
lowest order term in the expansion
(\ref{eq:section}). For every $k > 0$ we then define the subset
$P_k\subset \mathbf{Z}^n$ given by
\[ P_k = \left\{ \nu(f)\, :\, f\in R_k\right\}\subset
\mathbf{Z}^n.\]
The Okounkov body is defined to be  the closure
\[ P = \overline{\bigcup_{k\geqslant 1} \frac{1}{k} P_k}.\]
The property that $\nu(fg)=\nu(f)+\nu(g)$ can be used to show that
$P$ is a convex body in the positive orthant of $\mathbf{R}^n$. Let us
write $\Delta_\epsilon\subset\mathbf{R}^n$ for the $n$-simplex
\[ \Delta_\epsilon = \{(a_1,\ldots,a_n)\,:\, a_i\geqslant 0, \sum
a_i\leqslant \epsilon\}.\]
It will be useful to know that $P$ contains $\Delta_\epsilon$ for small
$\epsilon$ and for this it is important that we are using the graded
lexicographic order and not the ungraded version. 
\begin{lem}\label{lem:simplex}
	For sufficiently small $\epsilon>0$ we have $\Delta_\epsilon
	\subset P$. More precisely there exists some $\epsilon > 0$ such
	that for sufficiently large $k$ we have $
	\Delta_{k\epsilon-1}\cap\mathbf{Z}^n \subset P_k$. 
\end{lem}
\begin{proof}
	Let $\epsilon > 0$ be a small rational number, smaller than the
	Seshadri constant of $p$ with respect to $L$ (in other words
	the $\mathbf{Q}$-line bundle $L-\epsilon E$ on the blowup 
	$Bl_pX$ is ample). Let $\mathcal{I}_p$ be the ideal sheaf of
	$p$. If $k$ is such that $k\epsilon$ is an
	integer, consider the exact sequence
	\[ 0\longrightarrow \mathcal{I}_p^{k\epsilon}L^k 
	\longrightarrow L^k
	\longrightarrow \mathcal{O}_{k\epsilon p}\otimes L^k|_p
	\longrightarrow 0.\]
	For large $k$ the cohomology group
	$H^1(X,\mathcal{I}_p^{\epsilon k}L^k)$ vanishes, so the map
	\[ H^0(X,L^k) \longrightarrow H^0(X,\mathcal{O}_{k\epsilon
	p}\otimes L^k|_p) \]
	is surjective. On the other hand this simply maps a section
	of $L^k$ to its $(k\epsilon-1)$-jet at $p$. It follows that for
	any $n$-tuple $\mathbf{a}=(a_1,\ldots,a_n)\in \mathbf{Z}^n$ with 
	$a_i\geqslant 0$ and $\sum a_i\leqslant k\epsilon- 1$ there
	exists a section $f\in H^0(X,L^k)$ such that
	$\nu(f)=\mathbf{a}$. This implies that the Okounkov body $P$
	contains $\Delta_\epsilon$.  
\end{proof}

Now suppose that we have a filtration $\{F_iR\}$ on $R$ as in
Definition~\ref{defn:filtration}. 
Boucksom-Chen~\cite{BC09} showed how this gives rise
to a convex function on the Okounkov body (or concave in their case,
since our conventions differ).
Briefly the construction goes as follows. For every $t\geqslant 0$ we
can define a graded subalgebra $R^{\leqslant t}\subset R$ whose degree
$k$ piece is 
\begin{equation}\label{eq:Rleqt}
	R^{\leqslant t}_k = F_{\lfloor tk\rfloor}R_k.
\end{equation}
Using only sections of $R^{\leqslant t}$ we can repeat the construction
of the Okounkov body, and we will obtain a closed convex subset
$P^{\leqslant t}\subset P$, which will be non-empty as long as $t > t_0$
for some constant $t_0$. The convex transform of the filtration is
defined to be the function $G:P\to\mathbf{R}$ given by
\[ G(x) = \inf\{ t\,:\, x\in P^{\leqslant t}\}. \]
Then $G$ is convex, because of the following convexity
property:
\[ tP^{\leqslant s_1} + (1-t)P^{\leqslant s_2} \subset 
P^{\leqslant ts_1 + (1-t)s_2}.\]
It follows that $G$ is continuous on the interior of $P$, and in
\cite{BC09} it is shown that $G$ is lower semicontinuous on the whole of
$P$. The restriction of $G$ to the simplex $\Delta_\epsilon$ from
Lemma~\ref{lem:simplex} is also upper semicontinuous (see
Gale-Klee-Rockafellar~\cite{GKR}), so in fact $G$ is continuous near
the corner $0\in P$. 

We can arrive at the convex function $G$ in a slightly different way
too. Namely for each $k$, we let $G_k:P\to\mathbf{R}$ be the convex
envelope of the function 
\begin{equation}\label{eq:gk}
	\begin{aligned}
	g_k  : \frac{1}{k}P_k &\to \mathbf{R} \\
	\alpha &\mapsto \min\{ i/k\,:\, \text{there is }f\in F_iR_k\text{
	such that }\nu(f) = k\alpha\},
\end{aligned}\end{equation}
where we can let $G_k=\infty$ outside the convex hull of
$\frac{1}{k}P_k$. 
It can then be shown that $G_k \geqslant G$ for all $k$, and $G_k\to G$
uniformly on compact subsets of the interior of $P$, but $G_k$ might not
converge to $G$ on the boundary of $P$. 

A crucial point (see~\cite{Ny10}) is that for each $k > 0$ and any
function $T$ we have
\begin{equation}\label{eq:sumT}
	\sum_{i\geqslant 1} T(i/k)
\cdot(\dim F_iR_k - \dim F_{i-1}R_k) =
\sum_{\alpha\in
\frac{1}{k}P_k} T(g_k(\alpha)).
\end{equation}
In particular, if the filtration comes from a test-configuration, and
we write $A_k$ for the generator of the induced $\mathbf{C}^*$-action on
on the sections over the central fiber, then 
\begin{equation}\label{eq:Trgk}
	\mathrm{Tr}(A_k) = \sum_{i\geqslant 1} -i\cdot(\dim F_iR_k - \dim
F_{i-1}R_k) = -k\sum_{\alpha\in\frac{1}{k} P_k} g_k(\alpha).
\end{equation}
At the same time for continuous $T$, we have the asymptotic result
\begin{equation}\label{eq:limT}
	\lim_{k\to\infty} \frac{1}{k^n} \sum_{\alpha\in\frac{1}{k}P_k}
T(g_k(\alpha)) = \int_P T\circ G\,d\mu,
\end{equation}
where $\mu$ is the Lebesgue measure on $P$. This shows for instance that
if $\chi$ was induced by a test-configuration, then in the expansions
\eqref{eq:expand} we have
\begin{equation}\label{eq:a0b0}
	a_0 = \mathrm{Vol}(P), \quad
	b_0 = -\int_P G_\chi\,d\mu,
\end{equation}
where $G_\chi$ is the convex transform of the filtration $\chi$. Note
that the coefficients $a_1$ and $b_1$ cannot be expressed in terms of
the Okounkov body and the convex transform in general. This is only
possible for very special filtrations, for example the filtrations on
toric varieties that we discussed in Section~\ref{sec:toric}. 

We will often start with a filtration $\chi$, and look at the
corresponding sequence of test-configurations $\chi^{(k)}$ obtained from
the induced filtration on $R_k$. The following lemma gives some simple
properties of the corresponding convex transforms. 
\begin{lem}\label{lem:filtapprox}
	Let $\chi$ be a filtration on $R$, and for each $k$, let
	$\chi^{(k)}$ be the test-configuration given by the filtration
	on $R_k$. Let us also write $\chi^{(k)}$ for the corresponding
	filtration that we defined in Section~\ref{sec:filt}, which is
	canonically defined on the Veronesi subalgebra
	$\bigoplus_{i\geqslant 0} R_{ki}$. For each
	$l$ we can then construct functions 
	\[ g_l, g^{(k)}_l : \frac{1}{l}P_l \to \mathbf{R}, \]
	according to \eqref{eq:gk}, and also we have the concave
	transforms $G, G^{(k)}$. These functions satisfy the following
	properties:
	\begin{enumerate}
		\item We have $g^{(k)}_k = g_k$, and
			$g^{(k)}_{kl}\geqslant g_{kl}$ for each $k,l$. 
		\item If the filtration $\chi$ satisfies 
			$R_1\subset F_NR$, then $g^{(k)}_{kl}\leqslant
			N$ for all $k,l$. In addition
			$G^{(k)}\leqslant N$ for each $k$. 
		\item $G^{(k)}\geqslant G$ for all $k$, and 
			$G^{(k)}\to G$ uniformly on compact subsets of
			the interior of $P$. 
	\end{enumerate}
\end{lem}
\begin{proof}
	Let $F_iR$ be the filtration $\chi$, and for a fixed $k$ 
	write $F_i'R$ for the filtration $\chi^{(k)}$. Then by the
	construction of $\chi^{(k)}$ we have $F_i'R_k = F_i R_k$ for
	each $i$ since the filtrations on $R_k$ induced by $\chi$ and
	$\chi^{(k)}$ conincide. In addition, for each $l > 1$ and $i$,
	$F_i'R_{kl}$ is the smallest possible subspace, such that the
	multiplicative property holds for the filtration $\chi^{(k)}$.
	It follows that 
	\begin{equation}\label{eq:FiRkl}
		F_i'R_{kl} \subset F_iR_{kl}\,\text{ for each }i, l
		\geqslant 1.
	\end{equation}
	We now prove the 3 statements that we need. 
	\begin{enumerate}
		\item Since $F_i'R_{kl} \subset F_i R_{kl}$ for all
			$i,l\geqslant 1$, we have
			$g^{(k)}_{kl} \geqslant g_{kl}$. In addition
			equality holds for $l=1$ since $F_i'R_k =
			F_iR_k$ for all $i$. 
		\item If $R_1\subset F_NR$, then the multiplicative
			property implies $R_k\subset F_{kN}R$. On $R_k$
			the fitrations $\chi^{(k)}$ and $\chi$ coincide,
			so we also have $R_k\subset F_{kN}'R$. Using the
			multiplicative property again, $R_{kl}\subset
			F_{klN}'R$. This implies that
			$g^{(k)}_{kl}\leqslant N$ for all $k, l$. At the
			same time, using the notation \eqref{eq:Rleqt}
			for the filtration $\chi^{(k)}$ we have
			$R_{kl}^{\leqslant N} = R_{kl}$, so from the
			construction of the convex transform $G^{(k)}$
			we have $G^{(k)}\leqslant N$. 
		\item	The fact that $G^{(k)}\geqslant G$ follows from
			\eqref{eq:FiRkl} and the definition of the
			convex transform. Moreover $G^{(k)}$ is bounded
			above by the convex envelope of $g^{(k)}_k =
			g_k$, but on compact subsets of the interior of
			$P$, the convex envelopes of $g_k$ converge to
			$G$ as $k\to\infty$. 
	\end{enumerate}
\end{proof}

One consequence is the following formula for the norm of a filtration
$\chi$.
\begin{lem}\label{lem:norm}
	Given a filtration $\chi$, its norm $\Vert\chi\Vert_2$
	can be expressed in
	terms of the convex transform $G_\chi$ as follows:
\begin{equation}\label{eq:norm}
	\Vert\chi\Vert^2_2 = \int_P (G_\chi - \overline{G}_\chi)^2\,d\mu,
\end{equation}
where $\overline{G}_\chi$
is the average of $G_\chi$ on $P$. 
\end{lem}
\begin{proof}
	Recall that we defined the norm $\Vert\chi\Vert_2$ by
	approximating $\chi$ using finitely generated filtrations
	$\chi^{(k)}$, induced by the filtration $\chi$ on $R_k$.
	Let us write
	$c_0^{(k)}$ for the constant in the
	expansion \eqref{eq:TrA2} corresponding
	to the test-configuration $\chi^{(k)}$, and $G^{(k)}$ for the
	convex transform of $\chi^{(k)}$. From 
	\eqref{eq:sumT} and \eqref{eq:limT}
	applied to $T(x)=x^2$, we get	
	\[ c_0^{(k)} = \int_P (G^{(k)})^2\,d\mu.\]
	Using also the formulas analogous to \eqref{eq:a0b0} for
	$\chi^{(k)}$ and the definition of the norm in
	\eqref{eq:normchi}, we get
	\[ \Vert\chi^{(k)}\Vert_2^2 = \int_P (G^{(k)})^2\,d\mu -
	\frac{1}{\mathrm{Vol}(P)}\left(\int_P G^{(k)}\,d\mu\right)^2.\]
	By Lemma~\ref{lem:filtapprox} we have $G^{(k)}\to G_\chi$ uniformly
	on compact subsets of the
	interior of $P$, and also all the
	functions are uniformly bounded by the same constant. Therefore
	the formula \eqref{eq:norm} follows by letting $k\to \infty$. 
\end{proof}

It is important to note that the Okounkov body $P$ and the convex
transform $G_\chi$ will
in general depend on the point and local coordinates 
chosen in the construction of
the Okounkov body. The volume of $P$ and 
the integrals in \eqref{eq:a0b0} and \eqref{eq:norm} are however
independent of these choices. 

We record the following lemma, which we will use in the next section.
\begin{lem}\label{lem:Chowunstable}
	Suppose that $\chi$ is a filtration for $(X,L)$. 
	Write $G_\chi$ for the convex transform, and $g_k$ for the function
	defined in \eqref{eq:gk}. If
	\begin{equation}\label{eq:Chow1}
		\sum_{\alpha\in
		\frac{1}{k}P_k} g_k(\alpha) - \overline{G}_\chi\dim R_k < 0
	\end{equation}
	for infinitely many $k$, 
	then $(X,L)$ is asymptotically Chow unstable.  
\end{lem}
\begin{proof}
	As in Lemma~\ref{lem:filtapprox}, 
	consider the test-configuration $\chi^{(k)}$ 
	given by the induced filtration
	on $R_k$. Let us also write $A_{kl}$ for the generator of the
	$\mathbf{C}^*$-action on $R_{kl}$ given by the
	test-configuration $\chi^{(k)}$. 
	Writing $g^{(k)}_l$ for the functions
	corresponding to $\chi^{(k)}$ as in
	Lemma~\ref{lem:filtapprox}, we have
	\[ \mathrm{Tr}(A_{kl}) = -kl\sum_{\alpha\in\frac{1}{kl} P_{kl}}
	g^{(k)}_{kl}(\alpha),\]
	from Equation~\eqref{eq:Trgk}. From Lemma~\ref{lem:filtapprox}
	we then get
	\[ \mathrm{Tr}(A_{kl}) \leqslant -kl\sum_{\alpha\in\frac{1}{kl}
	P_{kl}} g_{kl}(\alpha),\]
	but crucially, equality holds for $l=1$. It then follows from
	Equation~\eqref{eq:limT}, that
	\[ \lim_{k\to\infty} \frac{1}{(kl)^{n+1}} \mathrm{Tr}(A_{kl})
	\leqslant -\int_P G_\chi\,d\mu.\]
	From the defining formula~\eqref{eq:Chow} for the Chow weight of
	this test-configuration, we get
	\[ \mathrm{Chow}_k(\chi^{(k)}) 
	\leqslant -\frac{k}{\mathrm{Vol(P)}}\int_P
	G_\chi\,d\mu + \frac{k}{\dim R_k}\sum_{\alpha\in\frac{1}{k} P_k}
	g_k(\alpha).\]
	Since this is the Chow weight of a test-configuration with
	exponent $k$, and by assumption this expression
	is negative for infinitely
	many $k$, it follows that $(X,L)$ is asymptotically Chow
	unstable. 
\end{proof}

\section{Extending Stoppa's argument}\label{sec:stoppa}
In this section
we will prove Theorem A, which we state again here. 

\begin{thm}\label{thm:main} 
	Suppose that $X$ admits a cscK metric in $c_1(L)$ and the
	automorphism group of $(X,L)$ is finite. If $\chi$ is a
	filtration such that $\Vert\chi\Vert_2 > 0$, then
	$\mathrm{Fut}(\chi) > 0$. 
\end{thm}
\begin{proof}
	We will first assume that the dimension $n > 1$. 
	Choose a point in $X$ and local coordinates so that we can
	construct the Okounkov body $P$ of $(X,L)$, and the convex
	transform $G_\chi$ of the filtration. 
	If $\Vert\chi\Vert_2 > 0$, then according to the formula
	\eqref{eq:norm}, the function $G_\chi$ is not constant. Let $M$
	be the essential supremum of $G_\chi$, and $\overline{G}_\chi$
	its average. Let us write
	\[ \Lambda = \frac{9}{10}M + \frac{1}{10}\overline{G}_\chi, \]
	and consider the subalgebra $R^{\leqslant\Lambda}\subset R$. As
	before, write $P^{\leqslant \Lambda}$ for the convex subset of
	$P$ obtained by performing the Okounkov body construction using
	only sections of $R^{\leqslant \Lambda}$. By
	the construction of $G_\chi$ and the choice
	of $\Lambda$, the subset $P^{\leqslant \Lambda}\subset P$ is a
	proper subset. It follows that
	\[ \lim_{k\to\infty} k^{-n}\dim R_k^{\leqslant\Lambda} <
	\lim_{k\to\infty} k^{-n}\dim R_k,\]
	since these limits are just the volumes of
	$P^{\leqslant\Lambda}$ and $P$. 
	In addition it is shown in \cite{BC09} that
	$R^{\leqslant\Lambda}$ contains an ample series (see
        Definition~\ref{defn:ample}). Applying
	Theorem~\ref{thm:sz} we find a point $p\in X$ and a number
	$\epsilon > 0$, such that
	\begin{equation}\label{eq:vanish}
		R^{\leqslant\Lambda}_k \subset H^0(X, L^k\otimes
		I_p^{\lceil k\epsilon\rceil}),
	\end{equation}
	for all $k$. 
	We can now go back and use the point $p$ and any choice of local
	coordinates to construct the Okounkov body $P$, noting that the
	statement \eqref{eq:vanish} is independent of these choices.
	We can also assume that $\epsilon$ is small enough such that the
	simplex $\Delta_\epsilon$ satisfies 
	$\Delta_\epsilon\subset P$ according to
	Lemma~\ref{lem:simplex}.  
	Note that in constructing the Okounkov body, the sections 
	$f\in R_k$ which vanish to order at least $\lceil
	k\epsilon\rceil$ at $p$ all satisfy 
	\[ \frac{1}{k}\nu(f)\in \overline{P\setminus \Delta_\epsilon},\]
	so the convex transform (constructed again with the new choice
	of $p$) satisfies
	\begin{equation}\label{eq:Glarge}
		G_\chi(x) \geqslant \Lambda\text{ for
		}x\in\Delta_\epsilon.
	\end{equation}

	Now consider the sequence of test-configurations
	obtained by restricting the filtration $\chi$ to $R_k$ for each
	$k$, and write $\chi^{(k)}$ for the corresponding filtrations.
	We will argue by contradiction, assuming that
	\begin{equation}\label{eq:infFut}
		\liminf_{k > 0} \mathrm{Fut}(\chi^{(k)}) = 0. 
	\end{equation}
	Following \cite{Sto08} the key step is to obtain from this
	a test-configuration
	for the blowup of $X$ at a suitable point. Let $\delta > 0$ be
	small. Then we can choose $k$ as large as we like, such that 
	$\mathrm{Fut}(\chi^{(k)}) < \delta$, and to simplify notation,
	we let $\eta =
	\chi^{(k)}$. Write $G_\eta$ for the convex transform of $\eta$.  
	Given the point $p$ and parameter $\epsilon$, we
	can consider the filtration induced by $\eta$ on
	the subalgebra
	\[ \bigoplus_{k\geqslant 0} H^0(X,L^k\otimes I_p^{\lceil
	k\epsilon\rceil}) \subset \bigoplus_{k\geqslant 0} R_k.\]
	If $\epsilon$ is rational and 
	less than the Seshadri constant of $p$ in
	$(X,L)$, then this gives rise to a filtration on the blowup
	$(Bl_pX, L - \epsilon E)$, where $E$ is the exceptional divisor.
	Our goal is to prove that if $\delta$ and $\epsilon$ are
	sufficiently small, then we can use
	Lemma~\ref{lem:Chowunstable} applied to this filtration to show that
	the blowup is not asymptotically Chow stable.
	This will give us the required contradiction, since by
	Arezzo-Pacard's result~\cite{AP06} the blowup admits a cscK
	metric for small $\epsilon$, and so is asymptotically Chow
	stable by Donaldson's result~\cite{Don01}. 

	To compute the expression \eqref{eq:Chow1} on the blowup, note
	that we can simply work on the part of the Okounkov body $P$
	given by $\overline{P\setminus\Delta_\epsilon}$. We want to show
	that the numbers
	\begin{equation}
		Ch_m = \sum_{\alpha\in\overline{P\setminus\Delta_\epsilon}\cap
		\frac{1}{m}P_m} g_m(\alpha) -
		\frac{\int_{P\setminus\Delta_\epsilon}
			G_\eta\,d\mu}{\mathrm{Vol}(P\setminus
		\Delta_\epsilon)}\dim H^0(X,L^m\otimes I_p^{\lceil
		m\epsilon\rceil})
	\end{equation}
	are negative for large $m$, where the functions
	$g_m$ are constructed from
	the filtration $\eta$ according to \eqref{eq:gk}. 
	We will focus on those $m$ for which
	$m\epsilon\in\mathbf{Z}$. 
	At this point is it convenient to introduce normalizations
	$\widetilde{G}_\eta = G_\eta - \overline{G}_\eta$, and 
	$\widetilde{g}_m
	= g_m - \overline{G}_\eta$, so that $\widetilde{G}_\eta$ has
	zero average. It is easy to see that we can then
	compute $Ch_m$ using $\widetilde{g}_m$ and $\widetilde{G}_\eta$,
	and we get the same formula:
	\begin{equation}\label{eq:Cm}
		Ch_m = \sum_{\alpha\in\overline{P\setminus\Delta_\epsilon}\cap
		\frac{1}{m}P_m} \widetilde{g}_m(\alpha) -
		\frac{\int_{P\setminus\Delta_\epsilon}
		\widetilde{G}_\eta\,d\mu}{\mathrm{Vol}(P\setminus
		\Delta_\epsilon)}\dim H^0(X,L^m\otimes I_p^{\lceil 
		m\epsilon\rceil}).
	\end{equation}
	Replacing $g_m$ by $\widetilde{g}_m$ corresponds to changing the
	$\mathbf{C}^*$-action on the test-configuration $\eta$ by an
	action with constant weights, and this leaves the Futaki
	invariant unchanged. The advantage is that now in the expansion
	\eqref{eq:expand} for $\eta$ 
	we have $b_0=0$, and $\mathrm{Fut}(\eta) =
	-b_1/a_0$, where $b_1$ is given by (see \eqref{eq:Trgk})
	\begin{equation}\label{eq:gmexpand}
		\sum_{\alpha\in\frac{1}{k}P_k} \widetilde{g}_m(\alpha) 
		= -b_1 m^{n-1} +
		O(k^{n-2}).
	\end{equation}
	At the same time from the Riemann-Roch Theorem we have
	\begin{equation}\label{eq:RRBlX}
		\dim H^0(X, L^m\otimes I_p^{\lceil m\epsilon\rceil} ) =
		(a_0 - \mathrm{Vol}(\Delta_\epsilon))m^n + O(m^{n-1}). 
	\end{equation}
	
	It will be useful to define two
	boundary pieces of $\Delta_\epsilon$, namely let
	$\partial_0\Delta_\epsilon$ consist of those faces which meet in
	the origin, and let $\partial_1\Delta_\epsilon$ be the remaining
	face. In addition we define a boundary measure $d\sigma$, which
	equals the Lebesgue measure on the faces in
	$\partial_0\Delta_\epsilon$, and is a scaling of the Lebesgue
	measure on the remaining face $\partial_1\Delta_\epsilon$, 
	such that the volume of each face
	is $\epsilon^{n-1}/(n-1)!$. 
	Using that $\widetilde{g}_m\geqslant \widetilde{G}_\eta$, we have
	\begin{equation}\label{eq:sum}
		\begin{aligned}
		 \sum_{\alpha\in\overline{P\setminus\Delta_\epsilon}\cap
		 \frac{1}{m}P_m}& \widetilde{g}_m(\alpha)= 
		 \sum_{\alpha\in \frac{1}{m}P_m}
		 \widetilde{g}_m(\alpha) -
		\sum_{\alpha\in(\Delta_\epsilon\setminus
		\partial_1\Delta_\epsilon)\cap\frac{1}{m}P_m}
		\widetilde{g}_m(\alpha) \\
		\leqslant& \sum_{\alpha\in \frac{1}{m}P_m} 
		\widetilde{g}_m(\alpha) -
		\sum_{\alpha\in(\Delta_\epsilon\setminus
		\partial_1\Delta_\epsilon)\cap\frac{1}{m}P_m}
		\widetilde{G}_\eta(\alpha) \\
		=& -m^n \int_{\Delta_\epsilon} \widetilde{G}_\eta\,d\mu
		+m^{n-1}\left(-b_1 -
		\frac{1}{2}\int_{\partial_0\Delta_\epsilon} 
		\widetilde{G}_\eta\,d\sigma +
		\frac{1}{2}\int_{\partial_1\Delta_\epsilon}
		\widetilde{G}_\eta\,d\sigma\right)\\ &+ O(m^{n-2}).
	\end{aligned}\end{equation}
	Here we used an Euler-Maclaurin type formula for the sum of
	$\tilde{G}_\eta$ 
	over lattice points, see for example
	Guillemin-Sternberg~\cite{GS06}. Note that the sign of the
	integral over $\partial_1\Delta_\epsilon$ is different because
	we need to compensate for the fact that the lattice points on
	$\partial_1\Delta_\epsilon$ are missing from the sum.

	It will now be convenient to write $M=\overline{G}_\chi +
	10\lambda$, and so $\Lambda = \overline{G}_\chi + 9\lambda$, where
	$G_\chi$ is the convex transform of the filtration we started
	with. From Lemma~\ref{lem:filtapprox}, $G_\eta\to G_\chi$
	uniformly on compact subsets of the interior of $P$ as
	$k\to\infty$, but also $G_\eta \geqslant G_\chi$, so if $k$ is
	chosen to be large enough, we have
	\begin{equation}\label{eq:ineqs}
		\begin{aligned} G_\eta(x)&\geqslant \overline{G}_\chi + 9\lambda
			\,\text{ for }x\in\Delta_\epsilon, \\
		\int_{\partial_1\Delta_\epsilon}G_\eta\,d\sigma &\leqslant
		(M+\delta)\mathrm{Vol}(\partial_1\Delta_\epsilon)
		= (\overline{G}_\chi + 10\lambda +
		\delta)\frac{\epsilon^{n-1}}{(n-1)!},
	\end{aligned}\end{equation}
	where we also used \eqref{eq:Glarge}. Since $\overline{G}_\eta\to
	\overline{G}_\chi$ as $k\to\infty$, we can choose $k$ large
	enough so that \eqref{eq:ineqs} implies
	\begin{equation}\label{eq:ineqstilde}
		\begin{aligned} \widetilde{G}_\eta(x)&\geqslant 
			9\lambda - \delta
			\,\text{ for }x\in\Delta_\epsilon, \\
			\int_{\partial_1\Delta_\epsilon}\widetilde{G}_\eta\,
			d\sigma &\leqslant
			(10\lambda +
		2\delta)\frac{\epsilon^{n-1}}{(n-1)!},
	\end{aligned}\end{equation}

	Using these bounds in \eqref{eq:sum}, we have, assuming
	$n\geqslant 2$ and $\delta$ is sufficiently small, 
	\begin{equation}\label{eq:sumgm}
		\begin{aligned}
		\sum_{\alpha\in\overline{P\setminus\Delta_\epsilon}\cap
		\frac{1}{m}P_m} \widetilde{g}_m(\alpha) \leqslant &
		-m^n\int_{\Delta_\epsilon}\widetilde{G}_\eta\,d\mu + m^{n-1}
		\left(\delta - \frac{4\lambda\epsilon^{n-1}}{(n-1)!} +
		\delta\frac{(n+2)\epsilon^{n-1}}{2(n-1)!}\right) \\
		&+ O(m^{n-2}) \\
		\leqslant & -m^n\int_{\Delta_\epsilon}\widetilde{G}_\eta
		\,d\mu +
		m^{n-1}\left(\delta -
		\frac{\lambda\epsilon^{n-1}}{(n-1)!}\right) +
		O(m^{n-2}). 
	\end{aligned}
	\end{equation}
	For the other term in the expression \eqref{eq:Cm} for $Ch_m$, we
	have (using that $\widetilde{G}_\eta$ has intergral zero)
	\begin{equation}\label{eq:Cm2}
		\begin{aligned}
			\frac{\int_{P\setminus\Delta_{\epsilon}}
			\widetilde{G}_\eta
			\,d\mu}{\mathrm{Vol(P\setminus\Delta_\epsilon)}}
			\Big[\mathrm{Vol}(P\setminus\Delta_\epsilon)m^n +
			O(m^{n-1})\Big] \geqslant -m^n\int_{\Delta_\epsilon}
			\widetilde{G}_\eta\,d\mu - C\epsilon^nm^{n-1},
		\end{aligned}
	\end{equation}
	for some $C$, at least for large enough $m$. Combining
	\eqref{eq:sumgm} and \eqref{eq:Cm2} in the formula \eqref{eq:Cm}
	we have
	\[ Ch_m \leqslant m^{n-1}\left(\delta -
	\frac{\lambda\epsilon^{n-1}}{(n-1)!} + C\epsilon^n\right) +
	O(m^{n-2}).\]
	Choosing $\epsilon$ sufficiently small, it follows that if
	$\delta$ is small enough (i.e. we chose $k$ large enough when
	setting $\eta=\chi^{(k)}$), then $Ch_m < 0$ for all large $m$.
	This concludes the proof, in the case when $X$ has dimension $n
	> 1$. 

	Suppose now that $n=1$. We then take the product of $X$ with any
	cscK manifold, which has finite automorphism group. For example
	we can take $Y = X\times X$, with the polarization $L_Y =
	\pi_1^*L\otimes \pi_2^*L$, where $\pi_1,\pi_2$ are the two
	projection maps. Writing $R^Y=\bigoplus R^Y_k$ for the
	homogeneous coordinate ring of $(Y,L_Y)$, we have $R^Y_k =
	R_k\otimes R_k$. A filtration $\chi$ 
	for $R$ naturally induces a
	filtration $\chi^Y$ for $R^Y$, simply by letting
	\[ F_iR^Y_k = (F_iR_k)\otimes R_k, \]
	for each $i,k$. Moreover this operation commutes with taking
	the sequence of finitely generated filtrations induced by a
	given filtration. In other words, the filtration
	$(\chi^{(i)})^Y$ coincides with the filtration
	$(\chi^Y)^{(i)}$. Now suppose that $\chi$ is given by a
	test-configuration, and $\chi^Y$ is the induced
	test-configuration for $Y$. Writing $A_k$ and $A_k^Y$ for the
	generators of the corresponding $\mathbf{C}^*$-actions, we can
	calculate that
	\[ \mathrm{Tr}(A^Y_k) = (\dim R_k)\mathrm{Tr}(A_k),\]
	and
	\[ \mathrm{Tr}( (A^Y_k)^2) = (\dim R_k)\mathrm{Tr}(A_k^2).\]
	From these it is straight forward to calculate that
	\[ \begin{aligned}
		\mathrm{Fut}(\chi^Y) &= \mathrm{Fut}(\chi) \\
		\Vert\chi^Y\Vert_2 &= \sqrt{a_0}\Vert\chi\Vert_2,
	\end{aligned}\]
	where $a_0$ is the volume of $(X,L)$ as usual. It follows that
	the $n=1$ case is a consequence of the $n=2$ case that we
	already proved. 
\end{proof}

We prove another similar result, where the Futaki invariant of a
filtration is replaced by the asymptotic Chow weight, which we define as
\begin{equation}\label{eq:asymptchow}
	\mathrm{Chow}_\infty(\chi) = \liminf_{k\to\infty}
\mathrm{Chow}(\chi^{(k)}). 
\end{equation}
Here as before, $\chi^{(k)}$ is the test-configuration induced by the
filtration $\chi$ by restricting $\chi$ to $R_k$. Note that if $\chi$ is
a finitely generated filtration, then because of \eqref{eq:chowinf} we
have $\mathrm{Chow}_\infty(\chi) = \mathrm{Fut}(\chi)$, but in  general
it is not clear what the relationship is between the two invariants. The
asymptotic Chow weight is the relevant notion for the definition of
$b$-stability in~\cite{Don10}.

\begin{prop}\label{prop:bstab}
	Suppose that $X$ admits a cscK metric in $c_1(L)$ and the
	automorphism group of $(X,L)$ is finite. Then if $\chi$ is a filtration
	for $(X,L)$ such that $\Vert\chi\Vert_2 > 0$, then
	$\mathrm{Chow}_\infty(\chi) > 0$. 
\end{prop}
\begin{proof}[Proof of Proposition~\ref{prop:bstab}]
The proof of this proposition is not too different from the proof of
Theorem~\ref{thm:main}. In fact we can follow the proof of
Theorem~\ref{thm:main} word for word up to Equation~\ref{eq:sum}, except
in Equation~\ref{eq:infFut} we use the Chow weight instead of the Futaki
invariant, and now we will have to control $Ch_m$ for $m=k$. In other
words we will not be able to take $m$ much larger than $k$, as was done
in the proof of Theorem~\ref{thm:main}. This makes the proof more
difficult and the convexity of the convex transform plays a crucial role
when we apply Lemma~\ref{lem:convexsum} below. 

Let us fix a small $\delta >
0$, and suppose initially that $n>1$. 
We can then find arbitrarily large $k$, such that the
test-configuration $\eta = \chi^{(k)}$ satisfies $\mathrm{Chow}(\eta) <
\delta$. As in the proof of Theorem~\ref{thm:main}, we introduce
normalized functions $\widetilde{G}_\eta = G_\eta - \overline{G}_\eta$, and
$\widetilde{g}_k = g_k - \overline{G}_\eta$. 
Then the Chow weight of $\eta$ is given by
\begin{equation}\label{eq:Choweta1}
	\mathrm{Chow}(\eta) = \frac{k}{\dim H^0(X,L^k)}\sum_{\alpha\in
	\frac{1}{k}P_k}
	\widetilde{g}_k(\alpha) < \delta. 
\end{equation}
Moreover using the notation from the proof of Theorem~\ref{thm:main}, if
we choose $k$ large enough, then we can assume that $\widetilde{G}_\eta$
satisfies similar bounds to \eqref{eq:ineqs}: 
\begin{equation}\label{eq:ineqs1}
	\begin{aligned} \widetilde{G}_\eta(x)&\geqslant 9\lambda -
		\delta\,\text{ for }x\in\Delta_\epsilon, \\
		\int_{\Delta_\epsilon\setminus\Delta_{\epsilon-n/k}}
		\widetilde{G}_\eta\,d\sigma &\leqslant
		(10\lambda + 2\delta)
		\mathrm{Vol}(\Delta_\epsilon\setminus\Delta_{
		\epsilon -n/k})
		\leqslant
		(10\lambda +
		2\delta)\frac{n\epsilon^{n-1}}{k(n-1)!},
	\end{aligned}
\end{equation}
for some $\lambda > 0$. As before, we want to control $Ch_k$, given by
the formula \eqref{eq:Cm}, with $k$ instead of $m$. We also have the
inequality \eqref{eq:Cm2} as before, so if $k$ is large enough, then 
\begin{equation}\label{eq:Ckineq}
	Ch_k \leqslant \sum_{\alpha\in\overline{P\setminus \Delta_\epsilon}\cap
	\frac{1}{k}P_k} \widetilde{g}_k(\alpha) + 
	k^n\int_{\Delta_\epsilon} \widetilde{G}_\eta\,d\mu +
	C\epsilon^nk^{n-1}. 
\end{equation}
In this equation we have
\begin{equation}\label{eq:sum2}
	\sum_{\alpha\in\overline{P\setminus\Delta_\epsilon}\cap
	\frac{1}{k}P_k} \widetilde{g}_k(\alpha)
	= \sum_{\alpha\in\frac{1}{k}P_k}\widetilde{g}_k(\alpha) - 
	\sum_{\alpha\in
	(\Delta_\epsilon\setminus\partial_1\Delta_\epsilon)\cap
	\frac{1}{k}P_k} \widetilde{g}_k(\alpha),
\end{equation}
and now we bound the last sum in a different way from what we did before, 
using
Lemma~\ref{lem:convexsum} below. Note that if $k$ is large enough, then 
by changing $\epsilon$ slightly, we can assume that
$k\epsilon\in\mathbf{Z}$. For example we can
replace $\epsilon$ by $\frac{1}{k}\lceil k\epsilon\rceil$ without
changing the last sum in \eqref{eq:sum2}. Then
\[
(\Delta_\epsilon\setminus \partial_1\Delta_\epsilon)\cap \frac{1}{k}P_k
= \Delta_{\epsilon - 1/k}\cap\frac{1}{k}P_k. \]
Using the bound \eqref{eq:ineqs1} together with 
Lemma~\ref{lem:convexsum} applied to the simplex
$\Delta_{\epsilon - 1/k}$, and that $\widetilde{g}_k\geqslant
\widetilde{G}_\eta$ on $\frac{1}{k}P_k$,  
we have 
\[ \sum_{\alpha\in \Delta_{\epsilon -1/k}\cap\frac{1}{k}P_k}
\widetilde{g}_k(\alpha)
\geqslant k^n\int_{\Delta_{\epsilon - n/k}} \widetilde{G}_\eta\,d\mu +
k^{n-1}(9\lambda-\delta)\frac{(3n-1)\epsilon^{n-1}}{2(n-1)!} 
- C_1k^{n-2}, \]
where we can choose $C_1$ to be independent of $\epsilon$ and $k$. 
Using \eqref{eq:ineqs1} again, we get
\[\begin{aligned}
	\sum_{\alpha\in \Delta_{\epsilon -1/k}\cap\frac{1}{k}P_k}
	\widetilde{g}_k(\alpha)
	\geqslant\, & k^n\int_{\Delta_{\epsilon}} \widetilde{G}_\eta\,d\mu -
k^{n-1}(10\lambda+2\delta)\frac{n\epsilon^{n-1}}{(n-1)!} \\
& +
k^{n-1}(9\lambda-\delta)\frac{(3n-1)\epsilon^{n-1}}{2(n-1)!} - C_1k^{n-2} \\
\geqslant &\, k^n\int_{\Delta_\epsilon} \widetilde{G}_\eta\,d\mu 
+ k^{n-1}\left(\frac{5\lambda}{2}-\frac{7n-1}{2}
\delta\right)\frac{\epsilon^{n-1}}{( n-1)!}
- C_1k^{n-2},
\end{aligned}
\]
where we used that $n\geqslant 2$. Putting this together with
\eqref{eq:sum2} into the bound \eqref{eq:Ckineq} for
$Ch_k$, if $\delta$ is sufficiently small we get
\[ Ch_k \leqslant \sum_{\alpha\in\frac{1}{k} P_k} \widetilde{g}_k(\alpha) -
k^{n-1}\left(2\lambda\frac{\epsilon^{n-1}}{(n-1)!} + C\epsilon^n\right)
+C_1k^{n-2}. \] 
Using the bound \eqref{eq:Choweta1} on the Chow weight of $\eta$, this
implies
\[ Ch_k\leqslant
k^{n-1}\left[\delta\mathrm{Vol}(P)-2\lambda\frac{\epsilon^{n-1}}{(n-1)!}
+ C\epsilon^n\right]
+ C_2k^{n-2},\]
where $C_2$ can be chosen to be independent of $\delta$. Now if we
choose $\epsilon$, and then $\delta$ sufficiently small,
then the leading coefficient is negative. So if $k$ is sufficiently
large we will have $Ch_k < 0$, and just as in Theorem~\ref{thm:main}, this
gives a contradiction. In addition just as before, the $n=1$ case can be
reduced to the higher dimensional result. 
\end{proof}

We used the following lemma. 
\begin{lem}\label{lem:convexsum}
	Suppose that for some rational $c\in (0,1)$, the function $f$ 
	is convex on the simplex
	\[ \Delta_c = \{(x_1,\ldots,x_n)\,:\, x_i\geqslant 0,\,
	x_1+\ldots+x_n\leqslant c\}\subset\mathbf{R}^n,\]
	and $f(x)\geqslant L$ for all
	$x\in\Delta_c$.
	There is a constant $C(n)$ depending only on the dimension such
	that for all large $k$ for which $kc\in\mathbf{Z}$ we
	have
	\[ \sum_{\alpha\in \Delta_c\cap\frac{1}{k}\mathbf{Z}^n} 
	f(\alpha) \geqslant
	k^n\int_{\Delta_{c-\frac{n-1}{k}}} f\,d\mu +
	k^{n-1}L \frac{(3n-1)c^{n-1}}{2(n-1)!} 
	- k^{n-2}C(n)L . \]
\end{lem}
With some more work it is likely that the integral can be taken over
$\Delta_c$, with a corresponding change in the $k^{n-1}$ term, but for us this
simpler result is enough. Such expansions for Riemann sums over
polytopes are well known (see e.g. Guillemin-Sternberg~\cite{GS06}), but
usually the error term depends on derivatives of the function. The point
of this result is that if $f$ is convex, then we have better control on
the error term. 
\begin{proof}
	First let us assume that $f\geqslant 0$. If $Q$ is a cube with
	volume $1/k^n$, then Jensen's inequality implies that
	\begin{equation}\label{eq:Jensen}
		\frac{1}{2^n}\sum_{v\text{ vertex of }Q} f(v) \geqslant
		k^n\int_Q f\,d\mu. 
	\end{equation}
	Now the key point is that we can cover the simplex
	$\Delta_{c-\frac{n-1}{k}}$ with cubes whose vertices are in
	$\Delta_c\cap\frac{1}{k}\mathbf{Z}^n$. Applying
	\eqref{eq:Jensen} to all of these cubes, we obtain
	\begin{equation}\label{eq:fgeq0}
		\sum_{\alpha\in \Delta_c\cap\frac{1}{k}\mathbf{Z}^n}
		f(\alpha) \geqslant k^n\int_{\Delta_{c-\frac{n-1}{k}}}
		f\,d\mu,
	\end{equation}
	since we will have to count each vertex at most $2^n$ times.
	Vertices near the boundary only need to be counted fewer times,
	but since $f \geqslant 0$, counting them more times just
	increases the sum. 

	In general if $f\geqslant L$, then we apply \eqref{eq:fgeq0} to
	$f-L$, and we get
	\begin{equation}\label{eq:genf}
		\sum_{\alpha\in\Delta_c\cap\frac{1}{k}\mathbf{Z}^n} f(\alpha)
	\geqslant k^n\int_{\Delta_{c-\frac{n-1}{k}}}f\,d\mu -
	k^nL\mathrm{Vol}(\Delta_{c-\frac{n-1}{k}}) + L\cdot
	\#(\Delta_c\cap\frac{1}{k}\mathbf{Z}^n), 
	\end{equation}
	where we know that the number of lattice points in
	$\Delta_c\cap\frac{1}{k}\mathbf{Z}^n$ is given by
	\[ \#(\Delta_c\cap\frac{1}{k}\mathbf{Z}^n) = k^n\frac{c^n}{n!} +
	k^{n-1}\frac{(n+1)c^{n-1}}{2(n-1)!} + O(k^{n-2}).\]
	At the same time
	\[ \mathrm{Vol}(\Delta_{c-\frac{n-1}{k}}) = \frac{c^n}{n!} -
	\frac{c^{n-1}}{k(n-2)!} + O(k^{-2}).\]
	Using these expansions in \eqref{eq:genf}, we get the required
	result. 
\end{proof}

\section{Relation to $b$-stability}\label{sec:bstability}

\subsection{Birationally transformed test-configurations}\label{sec:b-stab}
In~\cite{Don10}
Donaldson introduced a new notion of stability, called $b$-stability,
of which we quickly review one ingredient. See also~\cite{Don11}. 
The starting point is 
a test-configuration $\pi:(\mathcal{X},\mathcal{L})\to\mathbf{C}$ for
the pair 
$(X,L)$, and for simplicity we assume that the exponent of the
test-configuration is 1. In addition, suppose that the central fiber
$X_0$ has a
distinguished component $B$. 
Using this data, Donaldson defines a family of test-configurations
$(\mathcal{X}_i,\mathcal{L}_i)\to\mathbf{C}$. 
Given the same data, we can also define a 
filtration for the homogeneous coordinate ring, 
similarly to the construction of Witt Nystr\"om in
Equation~\ref{eq:filt1}. As
before, given any $s\in H^0(X,L^k)$ we extend this as a
$\mathbf{C}^*$-invariant meromorphic
section $\overline{s}$ of $\mathcal{L}^k$, and now we define for all $i,
k$
\begin{equation}\label{eq:bstab}
	F_i^BR_{k} = \{ s\in R_{k}\,:\, t^i\overline{s}\text{ has no pole at the
generic point of }B\}.
\end{equation}
We might need to modify the $\mathbf{C}^*$-action on
$\mathcal{L}$ by an action with constant weights to ensure that this
filtration satisfies $F_0R = \mathbf{C}$. Let us write $\chi$ for the
resulting filtration. The filtrations of $R_{k}$
for $k\geqslant 1$ induce a sequence of test-configurations
$\chi^{(k)}$, which
coincide with the birationally modified test-configurations defined by
Donaldson. 

One way to see this is using the point of
view of the Rees algebras. Let us write $F_iR$ for the filtration
corresponding to our test-configuration $\mathcal{X}$. Then we can think
of
\[ \bigoplus_{i\geqslant 0} (F_iR_k)t^i \]
as all the holomorphic sections of $\mathcal{L}^k$ over
$\mathcal{X}$, and 
\[ \bigoplus_{i\geqslant 0} (F_i^BR_k)t^i \]
as those meromorphic sections of $\mathcal{L}^k$, which only have poles
on $X_0\setminus B$. In the notation of \cite{Don11}, we can write this
as the sections of $\mathcal{L}^k\otimes\Lambda^m$ for some large enough
$m$, where $\Lambda^m$ is the sheaf of meromorphic functions with poles
of order at most $m$ along $X_0\setminus B$. In Donaldson's construction
we need to take sections $\overline{\sigma}_a$ which give a basis in
each fiber of
$\pi_*(\mathcal{L}^k\otimes\Lambda^m)$. These sections give an embedding
of $X\times\mathbf{C}^*$ into $\mathbf{P}^N\times\mathbf{C}$ where $\dim
R_k = N+1$, and the new family $(\mathcal{X}_k,\mathcal{L}_k)$ is the
closure of the image of this embedding. 
More explicitly, let us choose a decomposition of $R_k$
as a direct sum 
\[ R_k = \bigoplus R_{k,i},\]
where for each $i$ we have
\[ F_i^BR_k = \bigoplus_{j\leqslant i} R_{k,j}\]
Then choose a basis $\{\sigma_a\}$ for $R_k$ such that each $\sigma_a$
is in one of the $R_{k,i}$, i.e. $\sigma_a\in R_{k,i_a}$ for some $i_a$.
We can then define $\overline{\sigma}_a = t^{i_a}\sigma_a$ for each $a$.
Since these span the space of sections of
$\mathcal{L}^k\otimes\Lambda^m$ over the central fiber under the
restriction map
\[ \bigoplus_{i\geqslant 0} (F_i^BR_k)t^i \to \bigoplus_{i\geqslant 1}
(F_i^BR_k)/(F_{i-1}^BR_k),\]
they give a basis of sections for $\pi_*(\mathcal{L}^k\otimes\Lambda^m)$
at each point. The embedding of $X\times\mathbf{C}^*\to
\mathbf{P}^N\times \mathbf{C}$ is then given by
\[ (x, t) \mapsto ([t^{a_0}\sigma_0(x):\ldots:t^{a_N}\sigma_N(x)], t).
\]
The closure of this is precisely the test-configuration for $X$ given by
the $\mathbf{C}^*$-action with weights $a_0,\ldots,a_N$, which is the
same as the test-configuration given by the filtration $F_i^B$ on $R_k$.  
Therefore the sequence of birationally transformed test-configuration
$(\mathcal{X}_k,\mathcal{L}_k)$ coincides with our test-configurations
$\chi^{(k)}$. 

From this point of view, the main result of \cite{Don11} can be
rephrased as follows. Write $A_{k}$ for the generator of the
$\mathbf{C}^*$-action on the central fiber of the test-configuration
$\chi^{(k)}$, and let $N_{k}$ be the difference between the maximum
and minimum eigenvalues of $A_{k}$.
Then the result in \cite{Don11} is the following
\begin{thm}[Donaldson~\cite{Don11}]\label{thm:Donaldson}
	Suppose that $X$ admits a cscK
	metric in $c_1(L)$, and the automorphism group of $(X,L)$ is finite.
	Assume that central fiber $X_0$ above is reduced, and
	the component $B$ does 
	not lie in a hyperplane in $\mathbf{P}(H^0(X_0,L_0)^*)$. 
	Moreover, suppose that for each
	$k$, the power $\mathcal{I}_B^k$ of the ideal sheaf of $B$ in
	$\mathcal{X}$ coincides with the sheaf of holomorphic functions
	vanishing to order $k$ at the generic point of $B$.
	Then there
	is a constant $C > 0$, such that for all $k$ we have
	\begin{equation}\label{eq:chowlimit}
		\mathrm{Chow}(\chi^{(k)}) \geqslant Ck^{-1}N_{k}.
	\end{equation}
\end{thm}

It is natural to define a norm $\Vert\chi\Vert_{\infty}$ of
the filtration $\chi$ by
\[ \Vert\chi\Vert_{\infty} = \liminf_{k\to\infty} \frac{1}{k}N_k.\]
Then \eqref{eq:chowlimit} is equivalent to saying that if
$\Vert\chi\Vert_\infty> 0$, then $\mathrm{Chow}_\infty(\chi) > 0$, using
the asymptotic Chow weight we defined in Equation \eqref{eq:asymptchow}.

We will now show that
Proposition~\ref{prop:bstab} implies this theorem, even without the
condition on the powers $\mathcal{I}_B^k$ of the ideal sheaf of $B$.

\begin{prop}\label{prop:D2}
  Suppose that $X$ admits a cscK metric in $c_1(L)$, and the
  automorphic group of $(X,L)$ is finite. Suppose that we have a 
  test-configuration for $X$ with reduced central fiber $X_0$.
  Suppose that $X_0$ contains an irreducible component
  $B$, which is not contained in a hyperplane in $\mathbf{P}(H^0(X_0,
  L_0)^*)$. Construct the filtration $\chi$ as above.
  If $\Vert\chi\Vert_\infty > 0$, then $\mathrm{Chow}_\infty(\chi) > 0$. 
\end{prop}
\begin{proof}
We just need to show that $\Vert\chi\Vert_2 > 0$ in order to apply
Proposition~\ref{prop:bstab}. 
If $\Vert\chi\Vert_\infty > 0$, then the
test-configuration is necessarily non-trivial, and since $B$ is not
contained in any hyperplane the $\mathbf{C}^*$-action on $H^0(B,L_0)$ is
non-trivial (i.e. it does not have constant weights). We can choose a
$\mathbf{C}^*$-invariant complement of the space of sections vanishing
on $B$ inside $H^0(X_0, L_0^k)$. Let us write
\[ H^0(B,L_0^k)\subset H^0(X_0,L_0^k) \]
for this complementary subspace. By the construction,
the weights of the $\mathbf{C}^*$-action of the birationally modified 
test-configuration
$\chi^{(k)}$ on this subspace are the same as the weights of the original
test-configuration. Therefore the norm $\Vert\chi^{(k)}\Vert_2$ is
bounded below by the norm of the $\mathbf{C}^*$-action on $(B,L_0)$
given by the original test-configuration $\chi$.
So we just need to check that this $\mathbf{C}^*$-action
on $(B,L_0)$ has positive norm. Since $\chi$ is non-trivial, the corresponding
$\mathbf{C}^*$-action on $H^0(B,L_0)$ does not have constant weights, so the smallest
weight $\lambda_{min}$ differs from the largest weight
$\lambda_{max}$. Let $s_{min}$ and $s_{max}$ be corresponding
$\mathbf{C}^*$-equivariant sections. For any $k$ divisible by 3 we
have an inclusion
\[ H^0(B, L_0^{k/3}) \hookrightarrow H^0(B, L_0^k), \]
where the map is multiplication by $s_{min}^{2k/3}$. This implies that
in the weight decomposition of $H^0(B,L_0^k)$ there will be at least
$\dim H^0(B, L_0^{k/3})$ sections with weights at most
$\frac{k}{3}\lambda_{max} + \frac{2k}{3}\lambda_{min}$. Writing
$\lambda_k$ for the average weight on $H^0(B,L_0^{k/3})$ we then have
\[ \mathrm{Tr}\left[\left(A_k -
    \frac{\mathrm{Tr}(A_k)}{d_k}\right)^2\right] \geqslant c_0 k^n
\left(\lambda_k - \frac{k}{3}\lambda_{max} -
  \frac{2k}{3}\lambda_{min}\right)_+^2 \]
for some $c_0 > 0$, where we are writing $(x)_+ = \max\{x,0\}$.

In an similar way we can also get
\[
\mathrm{Tr}\left[\left(A_k -
    \frac{\mathrm{Tr}(A_k)}{d_k}\right)^2\right] \geqslant c_0 k^n
\left(\frac{2k}{3}\lambda_{max} +
  \frac{k}{3}\lambda_{min}- \lambda_k\right)_+^2
\]

Since
\[ (\lambda_k - a)_+^2 + (b - \lambda_k)_+^2 \geqslant
\frac{1}{2}(b-a)^2 \]
for any $a < b$, it follows that
\[ \mathrm{Tr}\left[\left(A_k -
    \frac{\mathrm{Tr}(A_k)}{d_k}\right)^2\right] \geqslant \frac{1}{2}
c_0k^{n+2}\left(\frac{\lambda_{max} -
    \lambda_{min}}{3}\right)^2. 
\]
In particular $\Vert\chi\Vert_2 > 0$, so we can apply
Proposition~\ref{prop:bstab}.
\end{proof}

\subsection{Filtrations from arcs}\label{sec:arcs}
In the definition of $b$-stability, in addition to families of
birationally modified test-configurations, one also needs to consider
more general degenerations which Donaldson calls arcs. 

Just like for
test-con\-fi\-gu\-ra\-tions, we first embed
$X$ into a projective space $X\subset \mathbf{P}^N$ using sections of $L^r$ for
some $r$. Then instead of acting by a one-parameter subgroup, we choose a
meromorphic map $g : D \to GL(N+1)$, where $D$ is the disk of radius 2
in $\mathbf{C}$
(by rescaling we could use any disk), 
such that $g$ restricts to
a holomorphic map on $\mathbf{C}^*$, and $g(1)=\mathrm{Id}$. 
Looking at the family $g(t)\cdot X$
for $t\not=0$, and taking the closure across zero in the Hilbert scheme,
we obtain a flat family $\pi:(\mathcal{X},\mathcal{L})\to D$, such
that the fibers away from $0$ are isomorphic to $(X,L^r)$. Conversely
any such family can be seen using a meromorphic map $g:D\to
GL(N+1)$ once it is embedded into a projective space. 

Such degenerations also give rise to filtrations in
a similar way to test-configurations. For simplicity we assume that
$r=1$. Thinking of a section $u\in H^0(L^k)$ as a section of
$\mathcal{L}^k$ over $\pi^{-1}(1)$, we can extend any section $u\in
H^0(L^k)$ to a meromorphic section $\overline{u}$ of $\mathcal{L}^k$
over $\mathcal{X}$. 
We define a filtration of 
$R=\bigoplus_{k\geqslant 0} H^0(L^k)$
by
\begin{equation}\label{eq:filtgt}
  F_iR = \left\{ u\in R\,:\, \begin{aligned}
      &\text{there exists a holomorphic
        family of sections }\\
      &v(t)\in R \text{ such that }\\ &t^i (\overline{u}
+ t\overline{v(t)})\text{ is holomorphic on
}\mathcal{X} \end{aligned}\right\},
\end{equation}

\noindent 
where to ensure that $F_0R=\mathbf{C}$, we may need to multiply $g(t)$ 
by a power of $t$. Note that if $t^i(\overline{u_1} +
t\overline{v_1(t)})$ and $t^j(\overline{u_2} + t\overline{v_2(t)})$
  are holomorphic, then so is their product
  \[ t^{i+j}( \overline{u_1u_2} + t\overline{u_1v_2(t) + u_2v_1(t) +
    tv_1(t)v_2(t)}), \]
so $u_1u_2\in F_{i+j}R$, and we get a filtration. 

This filtration $\chi$ gives rise to 
a sequence of test-configurations $\chi^{(k)}$ as usual.
More concretely, for each $k$, our arc induces a meromorphic family of
linear maps on $R_k = H^0(X,L^k)$, which we can think of as a
meromorphic
family of matrices $g_k(t)$, invertible for $t\not=0$. As 
explained in
\cite[Proposition 2]{Don10}, this family can be factored in the form
\begin{equation}\label{eq:factor}
  g_k(t) = L_k(t) t^{A_k} R_k(t),
\end{equation}
where $A_k$ is a diagonal matrix with entries $t^{\lambda_0},
t^{\lambda_1},\ldots, t^{\lambda_{N_k}}$, and $L_k(t), R_k(t)$ are
holomorphic and invertible for all $t$. We can then define a flag in
$R_k$, by letting $x\in F_i'R_k$ if $t^{A_k}$ acts on $R_k(0)x$ with
weights at least $-i$.

\begin{lem} The filtrations $F_i'$ and $F_i$ on $R_k$ defined using the factorization
  \eqref{eq:factor} and by \eqref{eq:filtgt} respectively coincide. 
\end{lem}
\begin{proof}
  We will do this for $k=1$, and we will drop the $k$ subscript. 
  For any $u\in R_1$, the extension $\overline{u}$ is just given by
  $g(t)u$. If $t^A$ acts on $R(0)u$ with weights at least $-i$, then
  $t^i L(t) t^A R(0)u$ is holomorphic, which means that
  \[ t^i g(t) R(t)^{-1}R(0)u \]
  is holomorphic. But $R(t) = R(0) + tS(t)$ for some holomorphic
  family of matrices $S(t)$, so
  \[ R(t)^{-1}R(0) u = u - tR(t)^{-1}S(t) u. \]
  Letting $v(t) = R(t)^{-1}S(t)u$ we see that $u\in F_iR_1$.

  Conversely suppose that we have $v(t)$ such that $t^i g(t)(u +
  tv(t))$ is holomorphic. Since
  \[ t^i g(t) (u + tv(t)) = t^i L(t)t^A R(0) u + t^{i+1} L(t) t^A
  \tilde{v}(t) \]
  for some $\tilde{v}(t)$, we see that $R(0) u$ cannot have a non-zero
  component in a weight space less than $-i$, since the resulting
  singularity cannot be cancelled using the other term. Therefore
  $u\in F_i'R_1$. 
\end{proof}

Given an arc, an extension of the Chow weight is defined
in~\cite{Don10}, which coincides with the usual Chow weight if the arc
is actually a test-configuration. We will see that this can be computed
from the filtration $\chi$. Let us take $r=1$ again for simplicity. 
We think of the degeneration as a
map $f:D\to \mathrm{Hilb}$, and pull back the Chow line bundle
$L_{\mathrm{Chow}}$ to $D$. Picking any element $x$ in the fiber over
$1$, we can use the map $g(t)$ to define a meromorphic section of
$L_{\mathrm{Chow}}$ over $D$, which is holomorphic away from the origin. 
If this section has a pole of order $-w$, then the Chow weight is
essentially $w$, once we normalize so that
each $g(t)$ is in $SL(N+1)$. To compute this, we just need
to know that according to Knudsen-Mumford~\cite{KM76}, the Chow line
bundle is the leading term $\lambda_{n+1}$ of the expansion
\begin{equation}\label{eq:KM}
	\det \pi_*(\mathcal{L}^k) = \lambda_{n+1}^{\binom{k}{n+1}}
	\otimes\ldots\otimes 
	\lambda_0,
\end{equation}
for large $k$, where $\lambda_i$ are certain natural $\mathbf{Q}$-line
bundles on the base of the family
$\pi:(\mathcal{X},\mathcal{L})\to D$ (in fact they are pulled back from
the Hilbert scheme, under the map $f$). 

In terms of the matrices $g_k(t)$ above, we are interested in the
asymptotics as $k\to\infty$ of the
order of the pole of $\det g_k(t)$ at $t=0$, where $g_1(t)$ is
normalized to be in $SL(N+1)$. From the factorization
\eqref{eq:factor} it is clear that the order of the pole is
$-\lambda_0 -\ldots -\lambda_N = -\mathrm{Tr}(A_k)$. The Chow weight
is then given up to a positive multiple by the asymptotic formula
\[ b_0 = \lim_{k\to\infty} k^{-(n+1)}\mathrm{Tr}(A_k). \]
If $g_1(t)$ were not normalized to be in $SL(N+1)$, then we could
compensate for this to get the general formula

\[ \widetilde{\mathrm{Chow}}_1(\chi) = \frac{b_0}{a_0} - \frac{w_1}{N+1},\]
where $a_0$ is the volume of $(X,L)$ as usual. This is analogous to the
formula we had in the case of a test-configuration, in
Equation~\ref{eq:Chow}. The subscript 1 means that the original
test-configuration had exponent 1 (in general the formula changes just
like for the usual Chow weight in Equation \eqref{eq:Chow}). 
In addition we put a tilde on top to distinguish this
Chow weight from the Chow weights $\mathrm{Chow}_k(\chi)$ of the
filtration in Definition~\ref{defn:filtFut}. 

In
general these two Chow weights are not equal, and in fact for each $k$
we have
\begin{equation}\label{eq:chowineq}
	\widetilde{\mathrm{Chow}}_k(\chi) \geqslant \mathrm{Chow}_k(\chi). 
\end{equation}
\noindent This is very similar to what we used in
Lemma~\ref{lem:Chowunstable}.
Indeed, focusing on the case when $k=1$, recall that
$\mathrm{Chow}_1(\chi)$ is the Chow weight of the test-configuration
induced by the filtration on $R_1$. As in Lemma~\ref{lem:filtapprox},
let us write $\chi^{(1)}$ for the corresponding finitely-generated
filtration. If we write $G_\chi$ and $G^{(1)}_\chi$ for the convex
transforms of $\chi$ and $\chi^{(1)}$ (corresponding to a fixed Okounkov
body), then the two Chow weights are
given by
\[\begin{aligned}
	\mathrm{Chow}_1(\chi) = -\overline{G}^{(1)}_\chi -
	\frac{w_1}{N+1}, \\
	\widetilde{\mathrm{Chow}}_1(\chi) = -\overline{G}_\chi -
	\frac{w_1}{N+1},
\end{aligned}\]
where we used the relations \eqref{eq:a0b0} for both $\chi$ and
$\chi^{(1)}$. 
From Lemma~\ref{lem:filtapprox} we know that $G^{(1)}_\chi \geqslant
G_\chi$, so the inequality \eqref{eq:chowineq} on the Chow weights
follows. It should not be surprising that we get a smaller Chow weight by
looking at the corresponding test-configuration, since by the
Hilbert-Mumford criterion we know that in
testing for Chow stability, it is enough to look at test-configurations
and we do not need general arcs. 

Let us now combine arcs with the construction from the previous section,
so let us suppose that we have a distinguished 
component $B$ in the central fiber of our arc $\mathcal{X}$. 
Just as in the case of
test-configurations, Donaldson constructs a sequence of arcs
$\mathcal{X}_i$. At the same time, we can also obtain a filtration
$\chi$ just like in Equation \eqref{eq:bstab}, by letting
\[
F_i^BR_{k} = \left\{ u\in R_{k}\,:\,\begin{aligned}
    &t^i(\overline{u}+t\overline{v(t)})
\text{ has no pole at } \\ &\text{the
  generic point of }B\text{ for some }v(t)\end{aligned}
\right\},
\]

Now the
sequence of test-configurations $\chi^{(i)}$ induced by $\chi$ are
certainly not the same as the arcs $\mathcal{X}_i$.
Instead for each $i$, the test-configuration $\chi^{(i)}$ is
simply the test-configuration given by the filtration on $H^0(X,L^i)$
which is induced by the arc $\mathcal{X}_i$. It follows then in the
same way as above, that the
Chow weight of the arc $\mathcal{X}_i$ is bounded from below by the Chow
weight $\mathrm{Chow}_i(\chi)$ of the test-configuration $\chi^{(i)}$.
In other words
\[ \liminf_{i\to\infty} \widetilde{\mathrm{Chow}}_i
(\mathcal{X}_i)\geqslant
\mathrm{Chow}_\infty(\chi),\]
in terms of the asymptotic Chow weight of the filtration. 

The conclusion from all this is that
Proposition~\ref{prop:bstab} can be used to obtain a result
analogous to Proposition~\ref{prop:D2} for arcs instead of
just test-configurations. 

\subsection{Webs of descendants}

The full definition of $b$-stability in \cite{Don10} 
focuses more on the possible
central fibers rather than the degenerations themselves.  This leads to
extra complications, since a given
scheme could be the central fiber of several different degenerations. It
is not clear whether filtrations are versatile enough to encode this
richer data of what Donaldson calls a ``web of descendants'', so we leave
a more detailed examination of this to future studies.

\newcommand{\A}{{\mathbb{A}}}
\newcommand{\C}{{\mathbb{C}}}
\newcommand{\N}{{\mathbb{N}}}
\newcommand{\PP}{{\mathbb{P}}}
\newcommand{\Q}{{\mathbb{Q}}}
\newcommand{\R}{{\mathbb{R}}}
\newcommand{\T}{{\mathbb{T}}}
\newcommand{\Z}{{\mathbb{Z}}}

\newcommand{\fa}{{\mathfrak{a}}}
\newcommand{\fb}{{\mathfrak{b}}}
\newcommand{\fp}{{\mathfrak{p}}}

\newcommand{\cA}{{\mathcal{A}}}
\newcommand{\cB}{{\mathcal{B}}}
\newcommand{\cC}{{\mathcal{C}}}
\newcommand{\cD}{{\mathcal{D}}}
\newcommand{\cE}{{\mathcal{E}}}
\newcommand{\cF}{{\mathcal{F}}}
\newcommand{\cH}{{\mathcal{H}}}
\newcommand{\cI}{{\mathcal{I}}}
\newcommand{\cJ}{{\mathcal{J}}}
\newcommand{\cL}{{\mathcal{L}}}
\newcommand{\cM}{{\mathcal{M}}}
\newcommand{\cN}{{\mathcal{N}}}
\newcommand{\cO}{{\mathcal{O}}}
\newcommand{\cS}{{\mathcal{S}}}
\newcommand{\cT}{{\mathcal{T}}}
\newcommand{\cV}{{\mathcal{V}}}
\newcommand{\cX}{{\mathcal{X}}}
\newcommand{\cW}{{\mathcal{W}}}

\newcommand{\ie}{i.e.~}

\newcommand{\psh}{{\mathrm{PSH}}}

\newcommand{\fA}{{\mathfrak{A}}}
\newcommand{\fB}{{\mathfrak{B}}}
\newcommand{\fC}{{\mathfrak{C}}}
\newcommand{\fD}{{\mathfrak{D}}}
\newcommand{\fE}{{\mathfrak{E}}}
\newcommand{\fF}{{\mathfrak{F}}}
\newcommand{\fH}{{\mathfrak{H}}}
\newcommand{\fI}{{\mathfrak{I}}}
\newcommand{\fJ}{{\mathfrak{J}}}
\newcommand{\fN}{{\mathfrak{N}}}
\newcommand{\fO}{{\mathfrak{O}}}
\newcommand{\fP}{{\mathfrak{P}}}
\newcommand{\fS}{{\mathfrak{S}}}
\newcommand{\fT}{{\mathfrak{T}}}
\newcommand{\fV}{{\mathfrak{V}}}
\newcommand{\fX}{{\mathfrak{X}}}
\newcommand{\fW}{{\mathfrak{W}}}
\newcommand{\fM}{{\mathfrak{M}}}
\newcommand{\fY}{{\mathfrak{Y}}}
\newcommand{\cK}{{\mathcal{K}}}

\newcommand{\tf}{\widetilde{\varphi}}
\newcommand{\tom}{\widetilde{\omega}}
\newcommand{\tm}{\widetilde{m}}
\newcommand{\tX}{\widetilde{X}}
\newcommand{\tV}{\widetilde{V}}

\renewcommand{\a}{\alpha}
\renewcommand{\b}{\beta}
\newcommand{\de}{\delta}
\newcommand{\e}{\varepsilon}
\newcommand{\om}{\omega}
\newcommand{\f}{\varphi}
\newcommand{\fm}{\varphi_{\min}}
\newcommand{\p}{\psi}
\newcommand{\la}{\lambda}
\newcommand{\s}{\sigma}
\newcommand{\D}{\Delta}

\newcommand{\FS}{\mathrm{f}}
\newcommand{\MA}{\mathrm{MA}}
\newcommand{\Amp}{\mathrm{Amp}\,}
\newcommand{\Hilb}{\mathrm{h}}
\newcommand{\Ric}{\mathrm{Ric}}
\newcommand{\scal}{\mathrm{Scal}}

\newcommand{\reg}{\mathrm{reg}}

\newcommand{\ddt}{\frac{d}{dt}}
\newcommand{\pddt}{\frac{\partial}{\partial t}}

\newcommand{\mes}{\mathrm{m}}
\newcommand{\tmes}{\widetilde{\mathrm{m}}}

\newcommand{\din}{\operatorname{Din}}
\newcommand{\mab}{\operatorname{Mab}}

\newcommand{\dbar}{\overline{\partial}}

\newcommand{\lc}{\operatorname{lc}}
\newcommand{\hi}{\operatorname{Hilb}}
\newcommand{\Hom}{\operatorname{Hom}}
\newcommand{\NS}{\operatorname{NS}}
\newcommand{\Aut}{\operatorname{Aut}}
\newcommand{\val}{\operatorname{val}}
\newcommand{\ord}{\operatorname{ord}}
\newcommand{\dist}{\operatorname{dist}}
\newcommand{\id}{\operatorname{id}}
\newcommand{\sing}{\operatorname{sing}}
\newcommand{\vol}{\operatorname{vol}}
\newcommand{\tr}{\operatorname{tr}}
\newcommand{\covol}{\operatorname{covol}}
\newcommand{\grad}{\operatorname{grad}}
\newcommand{\conv}{\operatorname{Conv}}
\newcommand{\res}{\operatorname{Res}}
\newcommand{\Spec}{\operatorname{Spec}}
\newcommand{\supp}{\operatorname{supp}}
\newcommand{\codim}{\operatorname{codim}}
\newcommand{\nef}{\operatorname{nef}}
\newcommand{\ca}{\operatorname{Cap}}

\newcommand{\eq}{{\mu_\mathrm{eq}}}
\newcommand{\eneq}{{E_\mathrm{eq}}}
\newcommand{\ena}{{\cE^\A_\mathrm{eq}}}
\newcommand{\ela}{{\cL^\A_k}}
\newcommand{\eler}{{\cL^\R_k}}

\section{Appendix: asymptotic vanishing orders of graded linear
  series -- S. Boucksom}
\subsection{Iitaka dimension and multiplicity}
Let $X$ be a projective variety over an algebraically closed field $k$
(of any characteristic), set $n:=\dim X$, and let $L$ be a line bundle
on $X$. Denote by  
$$
R=R(X,L):=\bigoplus_{m\in\N} H^0(X,mL)
$$ 
the algebra of sections of $L$. Given a graded subalgebra $S$ of $R$
(aka \emph{graded linear series} of $L$), set  
$$
\N(S):=\{m\in\N\mid S_m\ne 0\},
$$
which is a sub-semigroup of of $\N$, hence coincides outside a finite
set with the multiples of the gcd $m(S)\in\N$ of $\N(S)$, sometimes
known as the \emph{exponent} of $S$. Define also the \emph{Iitaka
  dimension} of $S$ as $\kappa(S):=\mathrm{tr. deg}(S/k)-1$ if $S\ne
k$, and $\kappa(S):=-\infty$ otherwise, so that
$\kappa(S)\in\{-\infty,0,1,...,n\}$.

In this generality, the following result is due to Kaveh and
Khovanskii \cite{KK} (see also \cite{Bou}).

\begin{thm}\label{thm:KK} Let $S\ne k$ be a graded subalgebra of
  $R(X,L)$, and write $\kappa=\kappa(S)$. \begin{itemize}  
\item[(i)] The \emph{multiplicity}
$$
e(S)=\lim_{m\in\N(S),\,m\to\infty}\frac{\kappa !}{m^\kappa}\dim S_m
$$
exists in $]0,+\infty[$. 
\item[(ii)] For each $m\in\N(S)$, let
  $\Phi_m:X\dashrightarrow\PP(S_m^*)$ be the rational map defined by
  linear series $S_m$, and denote by $Y_m$ its image. Then we have
  $\dim Y_m=\kappa$ for all $m\in\N(S)$ large enough, and  
$$
e(S)=\lim_{m\in\N(S),\,m\to\infty}\frac{\deg Y_m}{m^\kappa}.
$$
\end{itemize}
\end{thm}
Note that $L$ is big iff $\kappa(X,L):=\kappa(R)$ is equal to $n:=\dim
X$, and we then have $e(R)=\vol(L)$, the \emph{volume} of $L$.  

\begin{defn}\label{defn:ample}
  We say that $S$ \emph{contains an ample series} if 
\begin{itemize}
\item[(i)] $S_m\neq 0$ for all $m\gg 1$, \ie $S$ has exponent $m(S)=1$. 
\item[(ii)] There exists a decomposition $L=A+E$ into $\Q$-divisors
  with $A$ ample and $E$ effective such that $H^0(X,mA)\subset
  S_m\subset H^0(X,mL)$ for all sufficiently divisible $m\in\N$.  
\end{itemize}
\end{defn}
This condition immediately implies that the rational map
$\Phi_m:X\dashrightarrow\PP(S_m^*)$ defined by $S_m$ in birational
onto its image $Y_m$ for all $m\gg 1$.  

Assuming this, let $\fb_m\subset\cO_X$ be the \emph{base-ideal} of
$S_m$, \ie the image of the evaluation map
$S_m\otimes\cO_X(-mL)\to\cO_X$. Let $\mu_m:X_m\to X$ be any birational
morphism with $X_m$ normal and projective and such that
$\fb_m\cdot\cO_{X_m}$ is locally principal, hence of the form
$\cO_{X_m}(-F_m)$ for an effective Cartier divisor $F_m$ on $X_m$. We
then set 
\begin{equation}\label{equ:pm}
P_m:=\mu_m^*L-\tfrac 1 m F_m,
\end{equation}
which is a nef $\Q$-Cartier divisor on $X_m$. If $m$ divides $l$, then
we may choose $X_{l}$ to dominate $X_m$, and we have $P_{l}\ge P_m$
after pulling back to $X_l$ (in the sense that the difference is an
effective $\Q$-divisor). Note also that the intersection number
$(P_m^n)$ does not depend on the choice of $X_m$ by the projection
formula, and that $(P_{l}^n)\ge(P_m^n)$ when $m$ divides $l$, since
$P_m$ and $P_{l}$ are nef with $P_{l}\ge P_m$.  

As a consequence of Theorem \ref{thm:KK} above (see also \cite[Theorem
C]{Jow}), we get the following version of the Fujita approximation
theorem:  
\begin{cor}\label{cor:mult} Let $S$ be a graded subalgebra of $R$, and
  assume that $S$ contains an ample series. Then
  $e(S)=\lim_{m\to\infty}(P_m^n)$. 
\end{cor}
\begin{proof} With the notation of Theorem \ref{thm:KK}, the rational
  map $\Phi_m$ lifts to a morphism $f_m:X_m\to\PP(S_m^*)$ which is
  birational onto its image $Y_m$ and such that
  $f_m^*\cO(1)=\mu_m^*(mL)-F_m=mP_m$. We thus see that 
$$
(P_m^n)=\frac{\deg Y_m}{m^n},
$$ 
and the result follows from (ii) in Theorem \ref{thm:KK}.
\end{proof}

\begin{rem} The special case of Theorem \ref{thm:KK} where $S$
  contains an ample series, which is what is being used in the
  previous corollary, was first established in \cite{LM09}. 
\end{rem}

\subsection{Asymptotic vanishing orders and multiplicites}
Our goal is to prove the following result. 

\begin{thm}\label{thm:sz} Let $X$ be a smooth projective variety over
  an algebraically closed field $k$, and let $L$ be a line bundle on
  $X$. Let $S$ be a graded subalgebra of $R=R(X,L)$, and assume that
  $S$ contains an ample series. Assume also that
  $e(S)<e(R)=\vol(L)$. Then there exists $\e>0$ and a (closed) point
  $x\in X$ with maximal ideal $\mathfrak{m}_x\subset\cO_{X,x}$ such
  that $S_m\subset H^0\left(X,mL\otimes\mathfrak{m}_x^{\lfloor
      m\e\rfloor}\right)$ for all $m$.  
\end{thm} 

Recall that a \emph{divisorial valuation} (aka \emph{discrete
  valuation of rank $1$}) on $X$ is a valuation $v:k(X)^*\to\R$ of the
form $v=c\ord_E$ with $c>0$ and $E$ a prime divisor on a birational
model $X'$ of $X$, which can always be assumed to be normal,
projective and to dominate $X$. In particular, since $X$ is smooth,
every scheme theoretic point $\xi\in X$ defines a divisorial valuation
$\ord_\xi$. If we denote by $V=\overline{\{\xi\}}$ the subvariety of
$X$ having $\xi$ as its generic point, then we have for all
$f\in\cO_{X,x}$ 
\begin{equation}\label{equ:compdiv}
\ord_\xi(f)=\min_{x\in V}\ord_x(f).
\end{equation}
If we still denote by $\fb_m$ the base-ideal of $S_m$, then each
divisorial valuation $v$ on $X$ defines a subadditive sequence 
$$
v(\fb_m):=\min\{v(f)\mid f\in\fb_m\setminus\{0\}\},
$$ 
and we may thus define the \emph{asymptotic vanishing order of $S$
  along $v$} (cf. \cite{ELMNP}) as 
$$
v(S):=\lim_{m\to\infty}\frac{v(\fb_m)}{m}\in[0,+\infty[.
$$
In this language, the conclusion of Theorem \ref{thm:sz} amounts to
the existence of a closed point $x\in X$ such that $\ord_x(S)>0$. We
begin with the following consequence of Izumi's theorem on divisorial
valuations.  

\begin{lem}\label{lem:izumi} If there exists a divisorial valuation
  $v$ on $X$ such that $v(S)>0$, then $\ord_x(S)>0$ for some closed
  point $x\in X$.  
\end{lem}
\begin{proof} Let $\xi\in X$ be the center of $v$ on $X$ (concretely,
  there exists a birational morphism $\mu:X'\to X$ with $X'$
  projective and a prime divisor $E\subset X'$ such that $v=c\ord_E$,
  $c>0$, and $\xi$ is then the generic point of $\mu(E)\subset
  X$). Since the divisorial valuations $\ord_\xi$ and $v$ share the
  same center $\xi$ on $X$, the version of Izumi's theorem proved in
  \cite[Theorem 1.2]{HS} implies that there exists $C>0$ such that  
$$
C^{-1}v(f)\le\ord_\xi(f)\le C v(f)
$$
for all $f\in\cO_{X,\xi}$. Applying this to $f\in\fb_m$ yields in the
limit as $m\to\infty$  
$$
\ord_\xi(S)\ge C^{-1}v(S)>0.
$$ 
But for any closed point $x\in\overline{\{\xi\}}$ we also have
$\ord_x\ge\ord_\xi$ on $\cO_{X,x}$ by (\ref{equ:compdiv}), and this
similarly implies $\ord_x(S)\ge\ord_\xi(S)$, hence $\ord_x(S)>0$.  
\end{proof}

As a consequence of Corollary \ref{cor:mult}, we next prove: 
\begin{lem}\label{lem:mult} Let $S,S'$ be two graded subalgebras of
  $R$ containing an ample series. If $v(S)\ge v(S')$ for all
  divisorial valuations $v$, then $e(S)\le e(S')$.  
\end{lem}
\begin{proof} Let $\fb_m,\fb'_m\subset\cO_X$ be the base-ideals of
  $S_m$ and $S_m'$ respectively, and let $P_m$ and $P'_m$ be the nef
  $\Q$-Cartier divisors they determine on some high enough model $X_m$
  over $X$, as in (\ref{equ:pm}).  

Given $\e>0$, Corollary \ref{cor:mult} allows to find $m_0\in\N$ such that 
$e(S)\le(P_{m_0}^n)+\e$, and hence 
\begin{equation}\label{equ:m0}
e(S)\le(P_{m_1}\cdot P_{m_0}^{n-1})+\e
\end{equation}
for any multiple $m_1$ of $m_0$, since $P_{m_0}$ is nef and $P_{m_0}\le P_{m_1}$. 
By the projection formula and the definition of $P_{m_1}$ and $P'_{m_1}$, we have
$$
(P_{m_1}\cdot P_{m_0}^{n-1})-(P'_{m_1}\cdot
P_{m_0}^{n-1})=\sum_{E\subset
  X_{m_0}}\left(\frac{\ord_E(F'_{m_1})}{m_1}-\frac{\ord_E(F_{m_1})}{m_1}\right)(E\cdot
P_{m_0}^{n-1}), 
$$
where the sum runs over prime divisors $E$ of $X_{m_0}$ and any $E$
actually contributing to the sum is contained in the support of
$F_{m_0}+F'_{m_0}$, hence belongs to a finite set of prime divisors of
$X_{m_0}$ independent of $m_1$. Since we have by assumption 
$$
\lim_{m_1\to\infty}\frac{\ord_E(F_{m_1})}{m_1}=\ord_E(S)\ge\ord_E(S')=\lim_{m_1\to\infty}\frac{\ord_E(F'_{m_1})}{m_1} 
$$
for any such $E$, we may thus choose $m_1$ a large enough multiple of
$m_0$ to guarantee that  
$$
(P_{m_1}\cdot P_{m_0}^{n-1})\le(P'_{m_1}\cdot P_{m_0}^{n-1})+\e,
$$ 
and hence
\begin{equation}\label{equ:m1}
e(S)\le(P'_{m_1}\cdot P_{m_2}\cdot P_{m_0}^{n-2})+2\e
\end{equation}
for any multiple $m_2$ of $m_1$, by (\ref{equ:m0}) and the fact that
$P_{m_0}, P'_{m_1}, P_{m_2}$ are nef with $P_{m_0}\le P_{m_2}$. We
similarly have 
$$ \begin{aligned}
(P'_{m_1}\cdot P_{m_2}\cdot &P_{m_0}^{n-2})-(P'_{m_1}\cdot
P'_{m_2}\cdot P_{m_0}^{n-2})=\\ &\quad =\sum_{E\subset
  X_{m_1}}\left(\frac{\ord_E(F'_{m_2})}{m_2}-\frac{\ord_E(F_{m_2})}{m_2}\right)(P'_{m_1}\cdot
E\cdot P_{m_0}^{n-1})\le\e
\end{aligned} 
$$
for $m_2$ large enough, hence 
$$
e(S)\le(P'_{m_1}\cdot P'_{m_2}\cdot P_{m_3}\cdot P_{m_0}^{n-3})+3\e
$$
for any multiple $m_3$ of $m_2$, using (\ref{equ:m1}) and $P_{m_0}\le
P_{m_3}$. Continuing in this way, we finally obtain positive integers
$m_1,...,m_n$ with $m_i$ dividing $m_{i+1}$ and such that 
$$
e(S)\le(P'_{m_1}\cdot...\cdot P'_{m_n})+(n+1)\e, 
$$
hence
$$
e(S)\le(P'^n_{m_n})+(n+1)\e
$$
since $P_{m_i}\le P_{m_n}$. But $m_n$ can be taken to be as large as
desired, thus $(P'^n_{m_n})$ is as close to $e(S')$ as we like by
Corollary \ref{cor:mult}, and we conclude as desired that $e(S)\le
e(S')$. 
\end{proof}

\begin{proof}[Proof of Theorem \ref{thm:sz}] By Lemma \ref{lem:mult},
  the assumption $e(S)<e(R)$ implies that $v(S)>v(R)\ge 0$ for some
  divisorial valuation $v$. We conclude using Lemma \ref{lem:izumi}. 
\end{proof}

\bigskip
\noindent Department of Mathematics, University of Notre Dame, Notre
Dame, IN \\ {\tt gszekely@nd.edu}

\bigskip
\noindent CNRS-Universit{\'e} Pierre et Marie Curie,  I.M.J., F-75251 Paris Cedex 05,
 France \\ {\tt boucksom@math.jussieu.fr}

\end{document}